\documentclass[a4paper, 10pt]{amsart}

\usepackage[utf8]{inputenc}
\usepackage[T1]{fontenc}
\usepackage[english]{babel}
\usepackage{color}
\usepackage{amsmath} 
\usepackage{amssymb} 
\usepackage{amsthm}
\usepackage{graphicx}
\usepackage{hyperref}
\usepackage{cancel}

\usepackage{lmodern}
\usepackage{microtype}

\theoremstyle{plain}
\newtheorem{thm}{Theorem}[section]
\newtheorem*{thm*}{Theorem}
\newtheorem{prop}{Proposition}[section] 
\newtheorem{lem}{Lemma}[section] 
\newtheorem{cor}{Corollary}[section]
\newtheorem*{cor*}{Corollary}

\newtheorem{defi}{Definition}[section]

\newtheorem{rem}{Remark}

\newcommand {\R} {\mathbb{R}} \newcommand {\Z} {\mathbb{Z}}
\newcommand {\T} {\mathbb{T}} \newcommand {\N} {\mathbb{N}}
\newcommand {\C} {\mathbb{C}}
\newcommand {\p} {\partial}
\newcommand {\dt} {\partial_t}

\usepackage{cancel}
\usepackage{tikz}
\usetikzlibrary{shapes,arrows,calc}

\usepackage{enumerate}
\usepackage{pdfpages}

\begin{document}
 \thanks{This article is based on the author's PhD thesis written under the supervision of Prof. Dr. Herbert Koch, to whom I owe great gratitude for his constant advice, helpful direction and insightful comments.}
\thanks{I also thank the CRC 1060 ``On the Mathematics of Emergent Effects'' of the DFG, the Bonn International Graduate School (BIGS) and the Hausdorff Center of Mathematics for their support.}
\address{Mathematisches Institut, Universität Bonn, 53115 Bonn, Germany}

\date{\today}
\title[Linear inviscid damping and boundary effects]{Linear inviscid damping for monotone shear flows in a finite periodic channel, boundary effects, blow-up and critical Sobolev regularity}
\author{Christian Zillinger}

\maketitle

\begin{abstract}
In a previous article, \cite{Zill3}, we have established linear inviscid damping for a large class of monotone shear flows in a finite periodic channel and have further shown that boundary effects asymptotically lead to the formation of singularities of derivatives of the solution.
As the main results of this article, we provide a detailed description of the singularity formation and establish stability in all sub-critical fractional Sobolev spaces and blow-up in all super-critical spaces.
Furthermore, we discuss the implications of the blow-up to the problem of nonlinear inviscid damping in a finite periodic channel, where high regularity would be essential to control nonlinear effects.
\end{abstract}


\setcounter{tocdepth}{3}
\setcounter{secnumdepth}{3}
\tableofcontents

\section{Introduction}
\label{sec:introduction}

In this article, we are interested in the singularity formation of solutions to the linearized 2D Euler equations around monotone shear flows, $(U(y),0)$, in a finite periodic channel of period $L$, $\T_{L}\times [a,b]$,
\begin{align}
  \begin{split}
  \dt \omega + U(y) \p_{x} \omega &= U''(y) \p_{x}\phi=U''(y)v_{2}, \\
  \Delta \phi &= \omega , \\
  \p_{x}\phi|_{y=a,b}&=0, \\
  (t,x,y) &\in \R \times \T_{L} \times [a,b],
  \end{split}
\end{align}
as well as sharp stability results in fractional Sobolev spaces.

Here, we consider the transport by $(U(y),0)$ and the resulting shearing as the main underlying dynamics and hence introduce coordinates moving with the flow,
\begin{align*}
  W(t,x,y)&:=\omega(t,x-tU(y),y), \\
  \Phi(t,x,y)&:=\phi(t,x-tU(y),y).
\end{align*}
In these coordinates, the linearized Euler equations are given by 
\begin{align}
  \begin{split}
    \label{eq:LEusualcoord}
  \dt W &= U''(y)\p_{x}\Phi, \\
  (\p_{x}^{2}+(\p_{y}-tU'(y)\p_{x})^{2})\Phi &= W, \\
  \p_{x}\Phi|_{y=a,b}&=0, \\
  (t,x,y) &\in \R \times \T_{L} \times [a,b].
  \end{split}
\end{align}
In analogy to the conventions in dispersive equations, we call this a \emph{scattering formulation} of the linearized Euler equations with respect to the underlying transport problem 
\begin{align*}
  \dt f + U(y)\p_{x}f =0.
\end{align*}
As shown in \cite{Zill3}, under suitable assumptions on $U$ and $L$, one obtains linear inviscid damping and scattering:
\begin{thm}[Damping using regularity,{\cite[Theorem 2.1]{Zill3}}]
  \label{thm:lin-zeng}
  Let $U$ be such that $\frac{1}{U'}, U'' \in W^{2,\infty}(\T_{L}\times[a,b])$, let $W$ be the solution of 
  \eqref{eq:LEusualcoord} and let $v=\nabla^{\bot}\phi$ be the associated velocity field.
  Then the following statements hold:
  \begin{itemize}
  \item If $W(t)-\langle W\rangle_{x} \in H^{-1}_{x}H^{1}_{y}(\T_{L}\times [a,b])$, then $v$ satisfies
    \begin{align*}
      \|v(t)- \langle v\rangle_{x}\|_{L^{2}} = \mathcal{O}(t^{-1})\|W(t)- \langle W\rangle_{x}\|_{H^{-1}_{x}H^{1}_{y}},
    \end{align*}
    as $t \rightarrow \infty$.
  \item If $W(t)-\langle W\rangle_{x} \in H^{-1}_{x}H^{2}_{y}(\T_{L}\times [a,b])$, then $v_{2}$ satisfies
    \begin{align*}
      \|v_{2}(t)\|_{L^{2}} = \mathcal{O}(t^{-2})\|W(t)- \langle W\rangle_{x}\|_{H^{-1}_{x}H^{1}_{y}},
    \end{align*}
    as $t \rightarrow \infty$.
  \item If $W(t)-\langle W\rangle_{x} \in H^{-1}_{x}H^{s}_{y}(\T_{L}\times [a,b])$, for some $1<s<2$, then $v_{2}$ satisfies
    \begin{align*}
      \|v_{2}(t)\|_{L^{2}} = \mathcal{O}(t^{-s})\|W(t)- \langle W\rangle_{x}\|_{H^{-1}_{x}H^{s}_{y}},
    \end{align*}
    as $t \rightarrow \infty$.
  \end{itemize}
\end{thm}

\begin{thm}[Stability in $H^{1}$ and $H^{2}$, {\cite[Theorems 4.1, 4.3 and B.1]{Zill3}}]
\label{thm:citedH1H2result}  
  Let $U'(U^{-1}(\cdot), U''(U^{-1})(\cdot)) \in W^{3,\infty}([a,b])$ and suppose that 
  \begin{align*}
    0<c<U'(y)<c^{-1}<\infty,
  \end{align*}
  and that 
  \begin{align*}
    \|U''(U^{-1}(\cdot))\|_{W^{3,\infty}} L 
  \end{align*}
  is sufficiently small.
  Then, for any $\omega_{0} \in L^{2}_{x}H^{1}_{y}(\T_{L} \times [a,b])$, the solution of the linearized Euler equations, \eqref{eq:LEusualcoord}, with initial datum $\omega_{0}$ satisfies 
  \begin{align*}
    \|W(t)\|_{L^{2}_{x}H^{1}_{y}} \lesssim \|\omega_{0}\|_{L^{2}_{x}H^{1}_{y}}.
  \end{align*}
  
  Suppose that $U''|_{y=a,b} \neq 0$, then, for any $\omega_{0} \in L^{2}_{x}H^{2}_{y}(\T_{L} \times [a,b])$ with non-vanishing Dirichlet data, $\omega_{0}|_{y=a,b} \neq 0$, the solution $W$ satisfies 
  \begin{align*}
    \sup_{t\geq 0}\|W(t)\|_{L^{2}_{x}H^{2}_{y}(\T_{L}\times [a,b])} =\infty.
  \end{align*}
  
  Conversely, restricting to perturbations $\omega_{0} \in L^{2}_{x}H^{2}_{y}(\T_{L} \times [a,b])$ with vanishing Dirichlet data, $\omega_{0}|_{y=a,b}=0$, we obtain stability:
  \begin{align*}
    \|W(t)\|_{L^{2}_{x}H^{2}_{y}} \lesssim \|\omega_{0}\|_{L^{2}_{x}H^{2}_{y}}.
  \end{align*}
\end{thm}

Combining both theorems and restricting to perturbations with vanishing Dirichlet data, we thus obtain linear inviscid damping with the optimal rates,
\begin{align*}
  \|v(t)- \langle v \rangle_{x}\|_{L^{2}} &=\mathcal{O}(t^{-1})\|\omega_{0}- \langle \omega_{0} \rangle_{x}\|_{H^{-1}_{x}H^{1}_{y}}, \\
  \|v_{2}(t)\|_{L^{2}} &=\mathcal{O}(t^{-2})\|\omega_{0}- \langle \omega_{0} \rangle_{x}\|_{H^{-1}_{x}H^{2}_{y}},
\end{align*}
 as well as scattering, i.e. there exists $W^{\infty}$ such that
\begin{align*}
  W(t) \xrightarrow{L^{2}} W^{\infty},
\end{align*}
as $t \rightarrow \infty$.
\\

\subsection{Motivation}
\label{sec:motivation}

The necessity of requiring vanishing Dirichlet boundary values of the initial perturbation, $\omega_{0}$, in Theorem \ref{thm:citedH1H2result} is in sharp contrast to the common setting of an infinite periodic channel, $\T \times \R$, where stability holds in arbitrary Sobolev spaces (see \cite[Theorem 3.1]{Zill3}). 
In particular, it naturally raises the question whether damping with integrable rates and quantitative scattering results can be obtained for perturbations with non-vanishing Dirichlet data and what the sharp decay rates in this case are.

Furthermore, while $H^{2}\cap H^{1}_{0}$ stability is sufficient to establish linear inviscid damping with optimal rates, higher regularity is needed in order to prove consistency with the nonlinear equation, since the Sobolev embedding only yields an estimate of the form
\begin{align*}
  \|\nabla^{\bot}\Phi \cdot \nabla W\|_{L^{2}_{xy}} \lesssim \|\nabla^{\bot}\Phi\|_{L^{2}_{xy}}\|W\|_{H^{s}_{xy}}
\end{align*}
for (fractional) Sobolev spaces $H^{s},s>2$.

Additionally, as seen in the results of Bedrossian and Masmoudi, \cite{bedrossian2013asymptotic}, on nonlinear inviscid damping for Couette flow in the infinite periodic channel, very high regularity is needed to control nonlinear effects.
Stability and instability results for the linear dynamics and the associated blow-up in supercritical spaces are thus of great importance also for the problem of nonlinear inviscid damping in a finite periodic channel.

In this article, we hence study the effects of boundary conditions and the associated singularity formation in detail and aim at deriving optimal stability and blow-up results in fractional Sobolev spaces.

\subsection{Results for general perturbations}
\label{sec:general-results}

As the first main result of this article, we show that for general perturbations the fractional Sobolev space $H^{\frac{3}{2}}_{y}$ is critical in the sense that stability holds in all sub-critical fractional Sobolev spaces (Theorem \ref{thm:H32}) and (infinite-time) blow-up occurs in all super-critical fractional Sobolev spaces (Corollaries \ref{cor:not32} and \ref{cor:phiboundaryh2}) due to the formation of logarithmic singularities at the boundary.

\begin{cor}
Let $W$ be a solution of the linearized Euler equations, \eqref{eq:Eulernochmal}, and suppose that $U'(U^{-1}(\cdot)), U''(U^{-1}(\cdot)) \in W^{2,\infty}([0,1])$ and that $U'$ satisfies $(U')^{2}>c>0$. 
  Let $s>1$ and suppose that 
  \begin{align*}
    \|\p_{y} W\|_{H^{s}([0,1])}<C<\infty .
  \end{align*}
  Suppose further that $U''(U^{-1}(\cdot)) \omega_{0}|_{y=0,1}$ is non-trivial. Then  
  \begin{align*}
    \|\p_{y}W(t)\|_{L^{\infty}([0,1])} \gtrsim \log |t|
  \end{align*}
  as $t \rightarrow \pm \infty$.

As a consequence, for perturbations such that $U''(U^{-1}(\cdot)) \omega_{0}|_{y=0,1}$ is non-trivial, for any $2\geq s>\frac{3}{2}$, necessarily
  \begin{align*}
    \sup_{t>0} \|W(t)\|_{H^{s}([0,1])} = \infty .
  \end{align*}
\end{cor}

In our stability result we additionally assume periodicity in $y$. As we discuss in Remark \ref{rem:H32assumptionsperiod}, this is largely a technical assumption and the requirements can be relaxed.
\begin{thm}
  Let $0<s<1/2$, $\omega_{0} \in H^{1}([0,1])$ and $\omega_{0},\p_{y}\omega_{0} \in H^{s}(\T)$. Suppose further that $U'(U^{-1}(\cdot)), U''(U^{-1}(\cdot)) \in W^{2,\infty}(\T)$, that there exists $c>0$ such that 
  \begin{align*}
    0<c<U'<c^{-1}<\infty
  \end{align*}
  and that 
  \begin{align*}
    \|U''(U^{-1}(\cdot))\|_{W^{2,\infty}(\T)} L 
  \end{align*}
  is sufficiently small.
  Then the solution, $W$, of the linearized Euler equations, \eqref{eq:Eulernochmal}, satisfies 
  \begin{align*}
    \|\p_{y}W(t)\|_{H^{s}(\T)} \lesssim \|\omega_{0}\|_{H^{s}(\T)}+\|\p_{y}\omega_{0}\|_{H^{s}(\T)} ,
  \end{align*}
 uniformly in time.
\end{thm}

\subsection{Results for perturbations with vanishing Dirichlet data}
\label{sec:results-pert-with}

 When restricting to perturbations with vanishing Dirichlet data, $\omega_{0}|_{y=0,1}=0$, we similarly show that the critical Sobolev exponent is given by $\frac{5}{2}$ (Corollary \ref{cor:H2boundaryquadratic} and Theorem \ref{thm:H52}).
That is, generally $\p_{y}^{2}W$ asymptotically develops logarithmic singularities at the boundary, resulting in blow-up of all super-critical norms, while stability holds in all sub-critical fractional Sobolev spaces $H^{s},s<\frac{5}{2}$.

\begin{cor}
  Let $\omega_{0}|_{y=0,1}\equiv 0$ and let $W$ be the solution of \eqref{eq:Eulernochmal}. Further suppose that the limits
  \begin{align*}
    \lim_{t\rightarrow \infty} f(y)\p_{y}W|_{y=0,1}
  \end{align*}
exist (e.g. by Lemma \ref{lem:pWcontrol}) and are non-trivial.
Then for any $s>5/2$,
\begin{align*}
  \sup_{t\geq 0} \|W\|_{H^{s}} = \infty .
\end{align*}
\end{cor}

In our stability theorem we additionally require periodicity in $y$.
As discussed in Remark \ref{rem:H52assumptionsperiod}, this is largely a technical assumption and the requirements can be relaxed.
\begin{thm}
Let $0<s<1/2$ and let $\omega_{0} \in H^{2}([0,1])$, with vanishing Dirichlet data, $\omega_{0}|_{y=0,1}=0$, and $\omega_{0},\p_{y}\omega_{0},\p_{y}^{2}\omega_{0} \in H^{s}(\T)$. 
Suppose further that $U'(U^{-1}(\cdot)),$ $U''(U^{-1}(\cdot)) \in W^{3,\infty}(\T)$, that there exists $c>0$ such that
\begin{align*}
  0<c<U'<c^{-1}<\infty,
\end{align*}
and that 
  \begin{align*}
     \|U''(U^{-1}(\cdot))\|_{W^{3,\infty}(\T)} L 
  \end{align*}
is sufficiently small.
Then the solution, $W$, of the linearized Euler equations, \eqref{eq:Eulernochmal}, satisfies 
  \begin{align*}
    \|\p_{y}^{2}W(t)\|_{H^{s}(\T)} \lesssim \|\omega_{0}\|_{H^{s}}+ \|\p_{y}\omega_{0}\|_{H^{s}}+\|\p_{y}^{2}\omega_{0}\|_{H^{s}},
  \end{align*}
 uniformly in time.
\end{thm}

As consequences of these theorems, in Section \ref{cha:anoth-sect-cons}, we obtain linear inviscid damping (with integrable but not quadratic decay for $v_{2}$) as well as a quantitative scattering result for perturbations without zero Dirichlet data.

\subsection{Consistency and implications for the nonlinear problem}
\label{sec:cons-impl-nonl}

Using the stability results in $H^{s},s>2,$ we show that the linear evolution is consistent with the nonlinear equations (Theorems \ref{thm:EConsistency} and \ref{thm:EApproximation}).
Conversely, the nonlinear equations are shown to similarly asymptotically develop singularities on the boundary and hence exhibit blow-up in relatively low Sobolev regularity (Theorem \ref{thm:EInstability}):
\begin{thm}
  Let $(\omega,v)$ be a solution of the 2D Euler equations and define 
  \begin{align*}
    U(t,y)&:= \langle v_{1} \rangle_{x}(t,y), \\
    W(t,x,y)&:=\omega(t,x-\int_{0}^{t}U(t',y)dt',y), \\
    \Phi(t,x,y)&:=\phi(t,x-\int_{0}^{t}U(t',y)dt',y).
  \end{align*}
  Suppose further that 
  \begin{align*}
    \p_{y}^{2}U(t,y)|_{y=0} > c> 0,
  \end{align*}
  and that for some $k \in \Z$, 
  \begin{align}
    \label{eq:impliedbyinvisciddamping}
    \begin{split}
    \Re \mathcal{F}_{x}W(t,k,y)|_{y=0} &>c>0, \\
    |\mathcal{F}_{x}(\p_{y}(\nabla^{\bot}\Phi \cdot \nabla W))(t,k,y)|_{y=0}| &= \mathcal{O}(t^{-1-\epsilon}).
    \end{split}
  \end{align}
  Then, 
  \begin{align*}
    |(\mathcal{F}_{x}\p_{y}W)(t,k,y)|_{y=0}| \gtrsim \log(t),
  \end{align*}
  and as a consequence, for any $s>2$, 
  \begin{align*}
    \sup_{t>0}\|W(t)\|_{H^{s}} = \infty .
  \end{align*}
\end{thm}
As we discuss in Section \ref{cha:anoth-sect-cons}, the assumptions \eqref{eq:impliedbyinvisciddamping} can be shown to be satisfied, if the asymptotic shear flow is strictly monotone and $W(t)$ is sufficiently regular, i.e.
\begin{align*}
  0<c<\p_{y}U(t,y)<c^{-1}&< \infty,\\
  \|W(t)\|_{H^{6}}\leq C &<\infty,   
\end{align*}
as $t \rightarrow \infty$ and one further assumes that $\Re \mathcal{F}_{x}W(0,k,y)|_{y=0}$ is sufficiently large.
Regularity results for nonlinear inviscid damping in a finite periodic channel can thus in general not hold in high Sobolev regularity, which is in sharp contrast to the results of Bedrossian and Masmoudi, \cite{bedrossian2013asymptotic}, on nonlinear inviscid damping for Couette flow in an infinite periodic channel, where very high regularity is used to control nonlinear effects.
\\

\subsection{Outline of the article}
\label{sec:outline-article}

We conclude this introduction with a short overview of the structure of the article:
\begin{itemize}
\item In Section \ref{sec:linearized-2d-euler}, we introduce the linearized 2D Euler equations around monotone shear flows as well as some useful changes of coordinates. Furthermore, we briefly discuss the dynamics and the damping mechanism for the explicitly solvable case of Couette flow and a slightly more general constant coefficient model.
\item In Section \ref{sec:fract-sobol-space}, we introduce fractional Sobolev spaces, discuss some of their properties and introduce several estimates, which are used in the following sections.
\item In Section \ref{sec:stability-h32-}, we show that, for general perturbations, the critical Sobolev exponent is given by $\frac{3}{2}$, in the sense that stability holds for all sub-critical Sobolev spaces and blow-up occurs in all super-critical spaces. As a consequence we establish linear inviscid damping with damping rates integrable in time for initial data $\omega_{0}$ without vanishing Dirichlet data, $\omega_{0}|_{y=0,1}=0$.
\item In Section \ref{sec:stability-h52-}, we show that, for perturbations with vanishing Dirichlet data, the critical Sobolev exponent is improved to $\frac{5}{2}$.
  As we discuss in Section \ref{cha:anoth-sect-cons}, this regularity result can be used to prove consistency with the nonlinear Euler equations. Here, as a consequence of the super-critical blow-up, we also establish an instability result for the nonlinear dynamics, which, in particular, shows that results in Gevrey regularity such as in \cite{bedrossian2013asymptotic} can not be obtained in the setting of a finite channel.
\item In Section \ref{sec:boundary-layers}, we further study the singularity formation in a slightly simplified form and discuss critical stability and blow-up results in Sobolev spaces, $W^{s,p}$. 
\end{itemize}

\section{Preliminaries}
\label{sec:preliminaries}

\subsection{The linearized 2D Euler equations in scattering formulation}
\label{sec:linearized-2d-euler}

In this section, we introduce the linearized 2D Euler equations around monotone shear flows in a finite periodic channel. Subsequently, we employ multiple changes of variables and a Fourier transform in $x$ to simplify the equations.

The full 2D Euler equations in vorticity formulation in a finite periodic channel, $\T_{L}\times [a,b]$, with impermeable walls are given by 
\begin{align*}
  \dt \omega + v\cdot \nabla \omega &=0 , \\
  \Delta \phi &=\omega, \\
  v &= \nabla^{\bot}\phi, \\
  v_{2}|_{y=a,b}&=0, \\
(t,x,y)& \in \R \times \T_{L}\times [a,b].
\end{align*}
Considering solutions close to a shear flow, $(U(y),0)$, i.e. 
\begin{align*}
  v&= (U(y),0)+ v', \\
  \omega&= -U''(y))+ \omega',
\end{align*}
in the linearization we neglect the nonlinearity, 
\begin{align}
  v' \cdot \nabla \omega',
\end{align}
and thus obtain the \emph{linearized 2D Euler equations}: 
\begin{align}
  \label{eq:LE1}
  \begin{split}
    \dt \omega' + U(y)\p_{x}\omega' &= U''(y) v_{2}' , \\
    \Delta \phi'&=\omega', \\
    v_{2}' &= \p_{x}\phi', \\
    v_{2}'|_{y=a,b}&=0.
  \end{split}
\end{align}
Here, it is advantageous to introduce two changes of variables and a Fourier transform in $x$, in order obtain a more tractable formulation:
\begin{itemize}
\item We note that none of the coefficient functions, $U(y),U''(y)$, depend on $x$.
Hence, after a Fourier transform in $x$, the system \emph{decouples in frequency}:
\begin{align}
  \label{eq:LE2}
  \begin{split}
    \dt \hat{\omega}'(t,k,y) + U(y)ik \hat{\omega}'(t,k,y)&= U''(y)ik \hat{\phi}'(t,k,y),\\
    (-k^{2}+\p_{y}^{2})\hat{\phi}'(t,k,y)&=\hat{\omega}'(t,k,y), \\
    ik \hat{\phi}'(t,k,y)|_{y=0,1}&=0, \\
(t,k,y)& \in \R \times L\Z \times [0,1].
  \end{split}
\end{align}
In particular, the $x$ average, i.e. the mode $k=0$, is preserved in time.
Thus, using the linearity and modifying the initial data 
\begin{align*}
  \omega_{0} \mapsto \omega_{0}- \langle \omega_{0} \rangle_{x},
\end{align*}
we may without loss of generality restrict to modes $k\neq 0$.
\item Considering the case of Couette flow, i.e. $U(y)=y$, the linearized Euler equations reduce to a transport problem. In this case, one can hence compute the solution explicitly:
  \begin{align*}
    \omega'(t,x,y)&= \omega'(0,x+ty,y), \\
    \hat{\omega'}(t,k,y)&=\hat{\omega'}(0,k,y)e^{ikty}. 
  \end{align*}
  From this explicit form, we observe that it can not be expected that $\omega'(t,x,y)$ is regular with respect to $y$ uniformly in time, but that regularity can only be expected of the \emph{vorticity moving with the flow}
  \begin{align*}
    W(t,x,y)&=\omega'(t,x-tU(y),y), \\
    \hat{W}(t,k,y)&=\hat{\omega'}(t,k,y)e^{-iktU(y)}.
  \end{align*}
\item Using the Fourier transform and considering coordinates moving with the flow, the equation for the stream function 
  \begin{align*}
    \hat{\Phi}(t,k,y):= \hat{\phi'}(t,k,y)e^{-iktU(y)},
  \end{align*}
  is given by
  \begin{align*}
    (-k^{2}+(\p_{y}-iktU'(y))^{2})\Phi(t,k,y)=W(t,k,y).
  \end{align*}
  As $U'(y)$ is non-constant, an analysis of the behaviour of $(\p_{y}-iktU'(y))^{2}$ on frequency-localised functions would have to invest much technical effort to control error terms.
  It is thus advantageous to instead use that monotone shear flows are invertible and hence consider a change of variables 
  \begin{align*}
    (x,y) \mapsto (x,z)=(x,U^{-1}(y)).
  \end{align*}
\end{itemize}

Combing these three steps and introducing the notation 
\begin{align*}
  f(z):=U''(U^{-1}(z)), \\
  g(z):=U'(U^{-1}(z)),
\end{align*}
 we obtain the \emph{linearized 2D Euler equations in scattering formulation}:
\begin{align}
  \begin{split}
  \dt \hat W &= f(z) ik \hat\Phi, \\
  (-k^{2}+(g(z)(\p_{z}-ikt))^{2})\hat\Phi &= \hat W, \\
  \hat \Phi|_{z=U(a),U(b)}&=0, \\
  (t,k,z)&\in \R \times L\Z \times [U(a),U(b)].
  \end{split}
\end{align}
Here, the term scattering is used as in dispersive equations, i.e. the linearized Euler equations scatter with respect to the underlying transport equation 
\begin{align*}
  \dt f + U(y)\p_{x}f =0,
\end{align*}
which in the current coordinates means that $W(t,x,y)$ converges to an asymptotic profile $W^{\infty}(x,y)$ as $t \rightarrow \infty$. 

As additional simplifications for the notation, we relabel $z$ as $y$ and use the Galilean and scaling symmetries of the equations to reduce to the setting $[U(a),U(b)]=[0,1]$ (with $L$ rescaled by a factor as well).
Furthermore, since the system decouples in $k$, we consider $k$ as a given parameter, rescale $\hat{\Phi}$ by $k^{-2}$ and drop the hats, $\hat{\cdot}$, from our notation:
\begin{defi}[Linearized 2D Euler equations in scattering formulation]
Let $f,g: [0,1]\rightarrow \R$ be given. Then the \emph{linearized Euler equations in scattering formulation} are given by
\begin{align}
  \label{eq:LE}
  \begin{split}
\dt W &= \frac{if}{k}\Phi, \\
(-1+(g(\frac{\p_{y}}{k}-it))^{2})\Phi &= W, \\
   \Phi|_{y=0,1}&=0, \\
  (t,k,y)&\in \R \times L(\Z \setminus \{0\}) \times [0,1].
  \end{split}
\end{align}
\end{defi}

In the following section, we briefly discuss the damping mechanism for a simplified model in the setting of an infinite channel. There, an explicit solution allows us to clearly present the damping mechanism and discuss the challenges of deducing regularity results and the additional technical difficulties and boundary effects arising in the setting of a finite periodic channel.

\subsection{A constant coefficient model for the damping mechanism}
\label{sec:const-coeff-model}

In order to obtain some insights into the damping mechanism and the associated challenges, in the following we recall the \emph{constant coefficient model} from \cite{Zill3}.

In this model, we consider the linearized Euler equations in scattering formulation, \eqref{eq:LE}, in the case of an infinite periodic channel, $\T_{L} \times \R$ , and formally replace $f$ and $g$ by constants $c \in \C, d \in \R$:
\begin{align*}
  \dt \Lambda &= \frac{ic}{k} \Psi, \\
  (-1+d^{2}(\frac{\p_{y}}{k}-it)^{2})\Phi &= W, \\
(t,k,y)&\in \R \times L(\Z \setminus \{0\}) \times \R.
\end{align*}
We note that the case $c=0,d=1$ corresponds to Couette flow, $U(y)=y$.

As the coefficient functions are constant, after a Fourier transform in $y$, this system further decouples with respect to the frequency $\eta$ and is explicitly solvable.
In contrast to Couette flow, the dynamics of $\Lambda$ are however not trivial.
More precisely, we compute
\begin{align*}
  \dt \tilde \Lambda &= -\frac{ic}{k} \frac{1}{1+d^{2}(\frac{\eta}{k}-t)^{2}} \tilde \Lambda, \\
  \Rightarrow \tilde \Lambda (t,k,\eta)&= \exp\left(-\frac{ic}{k}\int_{0}^{t} \frac{1}{1+d^{2}(\frac{\eta}{k}-\tau)^{2}} d\tau \right) \tilde \Lambda (0,k,\eta).
\end{align*}
From this explicit calculation we observe multiple facts:
\begin{itemize}
\item In order to obtain decay of the multiplier in time, we need to require that $d^{2}>0$. In the case of the linearized Euler equations this corresponds to requiring $(U')^{2}\geq d^{2}>0$, i.e. strict monotonicity.
The solution operator is then uniformly bounded by 
\begin{align*}
  \exp \left(\frac{|c|}{|k|}\frac{\pi}{|d|}\right).
\end{align*}
\item The operator norm of $\Lambda \mapsto \Psi$ as mapping from $H^{s}$ to $H^{s}$ \emph{does not improve in time}, since 
  \begin{align*}
    \sup_{\eta} \frac{1}{1+d^{2}(\frac{\eta}{k}-t)^{2})} =1
  \end{align*}
is independent of time. This can also be seen more generally by noting that the change of variables $(x,y)\mapsto (x-tU(y),y)$ is a unitary operator on $L^{2}$ and hence conjugation with it does not change the operator norm.
\item One can additionally use that $|e^{i\Re(c)}|=1$. However, in the case of the linearized Euler equations, using this property corresponds to using anti-symmetry of operators and cancellations. As these are very fragile properties, we restrict ourselves to only using the more robust damping mechanism.
\end{itemize}
Since the linearized Euler equations do not admit explicit solutions, in our proof of stability in \cite{Zill3} we use a slightly more indirect argument. That is, we construct a decreasing Fourier weight 
\begin{align*}
  A(t) W := \mathcal{F}^{-1} \exp(C \arctan(\frac{\eta}{k}-t))\mathcal{F} W,
\end{align*}
where $C>0$, and show that, under suitable assumptions on $f,g$ and $L$,
\begin{align*}
  \frac{d}{dt} \langle W, A W \rangle \leq | 2\Re \langle W, A \frac{if}{k}\Phi \rangle | + \langle W, \dot A W \rangle \leq 0.
\end{align*}
Here, by our construction of $A$, the last inequality corresponds to an elliptic regularity result for $\Phi$.
Using that $A(t)$ is ``comparable to the identity'', i.e. 
\begin{align*}
  1 \lesssim \exp(C \arctan(\frac{\eta}{k}-t)) \lesssim 1,
\end{align*}
we thus obtain
\begin{align*}
  \|W(t)\|_{L^{2}}^{2} \lesssim \langle W,A(t)W \rangle \leq \langle \omega_{0}, A(0)\omega_{0} \rangle \lesssim \|\omega_{0}\|_{L^{2}}^{2}.
\end{align*}
The associated $L^{2}$ stability result for both the infinite and finite periodic channel is summarised in the following theorem:
\begin{thm}[{\cite[Theorems 3.4 and 4.2]{Zill3}}]
  \label{thm:citedL2result}
  Let $W$ be a solution of the linearized Euler equations, \eqref{eq:LE}, in either the infinite periodic channel, $\T_{L}\times \R$, or the finite periodic channel, $\T_{L}\times [0,1]$.
  Further suppose that there exists $c>0$ such that 
  \begin{align*}
    0<c<|g|<c^{-1}<\infty,
  \end{align*}
  and that 
  \begin{align*}
    \|f\|_{W^{1,\infty}}L
  \end{align*}
  is sufficiently small.
  Then, for any initial datum, $\omega_{0} \in L^{2}$, the solution $W$ satisfies 
  \begin{align*}
    \|W(t)\|_{L^{2}} \lesssim \|\omega_{0}\|_{L^{2}}.
  \end{align*}
\end{thm}
In the case of finite channel, this method of proof is shown to be very stable and to extend to stability results in arbitrary Sobolev spaces, $H^{s}, s \in \N$.
\\

When considering a finite channel, in addition to technical challenges such as finding a suitable replacement for a Fourier transform and for $A(t)$, one encounters boundary effects.
For simplicity, in the following we consider the example of linearized Couette flow on the channel $\T_{2\pi}\times [0,1]$ and initial datum $\omega_{0}(x,y)=2i \sin(x)$. 
The linearized Euler equations are then given by 
\begin{align*}
  \Lambda(t,1,y)&\equiv 1, \\
  (-1+(\p_{y}-it)^{2})\Psi &=1 \\
  \Psi|_{y=0,1}&=0.
\end{align*}
Taking one derivative in $y$, we observe that 
\begin{align*}
  \p_{y}\Lambda &\equiv 0, \\
  (-1+(\p_{y}-it)^{2})\p_{y}\Psi &= 0.
\end{align*}
The function $\p_{y}\Psi$ is a thus homogeneous solution, which in general has non-zero Dirichlet conditions and is hence non-trivial.
In particular, an estimate of $\p_{y}\Psi$ by $\p_{y}\Lambda$ can thus not hold.
In order to compute $\p_{y}\Psi|_{y=0,1}$ explicitly, one tests the equation for $\Psi$ with homogeneous solutions $e^{\pm y +ity}$:
\begin{align*}
  \langle 1, e^{\pm y +ity} \rangle_{L^{2}} = \langle (-1+(\p_{y}-it)^{2})\Psi, e^{\pm y +ity} \rangle_{L^{2}} = e^{\pm y +ity}\p_{y}\Psi|_{y=0}^{1}, 
\end{align*}
where we used that $e^{\pm y +ity}$ is a homogeneous solution and that $\Psi|_{y=0,1}=0$.
Considering suitable linear combinations,
\begin{align*}
  u_{1}(y)&:=-e^{ity}\frac{\sinh(1-y)}{\sinh(1)}, \\
  u_{2}(y)&:=e^{it(y-1)} \frac{\sinh(y)}{\sinh(1)},
\end{align*}
which have boundary values 
\begin{align*}
  -u_{1}(0)=u_{2}(1)&=1, \\
  u_{1}(1)=u_{2}(0)&=0,
\end{align*}
we thus obtain 
\begin{align}
\label{eq:Couetteboundaryexample}
  \begin{split}
  \p_{y}\Psi|_{y=0}&= \langle 1, u_{1} \rangle= -\frac{1}{it}e^{ity}\frac{\sinh(1-y)}{\sinh(1)}|_{y=0}^{1}- \frac{1}{it}\left\langle 1, e^{ity}\p_{y} \frac{\sinh(1-y)}{\sinh(1)}  \right\rangle = \frac{1}{it}+\mathcal{O}(t^{-2}) , \\
  \p_{y}\Psi|_{y=1}&= \langle 1, u_{2} \rangle= \frac{1}{it}e^{it(y-1)} \frac{\sinh(y)}{\sinh(1)} - \frac{1}{it} \left\langle 1, e^{it(y-1)}\p_{y} \frac{\sinh(y)}{\sinh(1)} \right\rangle = \frac{1}{it}+\mathcal{O}(t^{-2}).
  \end{split}
\end{align}
In particular, we note that, despite $\p_{y}W$ vanishing, $\p_{y}\Psi|_{y=0,1}$ only vanishes with a non-integrable rate.

Recalling the linearized Euler equations, \eqref{eq:LE}, taking a derivative in $y$ and restricting to the boundary, we observe that $\p_{y}W|_{y=0,1}$ satisfies
\begin{align*}
  \dt \p_{y}W|_{y=0,1} = \frac{if}{k} \p_{y}\Phi|_{y=0,1}. 
\end{align*}
Considering flows with $f|_{y=0,1}\neq 0$, the non-integrable decay rate in \eqref{eq:Couetteboundaryexample} thus suggests that $\p_{y}W|_{y=0,1}$ develops a (logarithmic) singularity as $t \rightarrow \infty$. 

In the following sections, we show that this singularity formation indeed occurs and obtain associated blow-up results in the fractional Sobolev spaces $H^{s},s>\frac{3}{2}$. Conversely, we show that stability holds in all sub-critical fractional Sobolev spaces, $H^{s},s<\frac{3}{2}$.
Furthermore, as can already partially be seen in \eqref{eq:Couetteboundaryexample}, the decay behaviour of $\p_{y}\Phi|_{y=0,1}$ improves if one restricts to initial perturbations, $\omega_{0}$, with vanishing Dirichlet data, $\omega_{0}|_{y=0,1}$. For such perturbations we show that the stability and blow-up results can be improved to $H^{s},s<\frac{5}{2},$ and $H^{s},s>\frac{5}{2}$, respectively.

\subsection{Fractional Sobolev spaces}
\label{sec:fract-sobol-space}

As we make extensive use of fractional Sobolev spaces, we provide a short introduction to their various definitions and properties.
Here we follow \cite{hitch} (published as \cite{hitchpublished}).

In the whole space, fractional Sobolev spaces can be equivalently characterized using either a Fourier weight or an appropriate kernel:
\begin{prop}[Fractional Sobolev space on $\R$; {\cite[Section 3]{hitch}}]
  Let $0<s<1$, then there exists $C_{s}$  such that for any $u \in \mathcal{S}(\R)$
  \begin{align*}
    \||\eta|^{s}\mathcal{F}u\|_{L^{2}}^{2} =  C_{s}\iint_{\R\times\R} \frac{|u(x)-u(y)|^{2}}{|x-y|^{1+2s}} dx dy . 
  \end{align*}
  In particular, both expressions define the same quasi-norm.
  The fractional Sobolev space, $H^{s}(\R)$, is then defined as the closure of $\mathcal{S}(\R)$ with respect to
  \begin{align*}
    \|u\|_{L^{2}}^{2}+ \||\eta|^{s}\mathcal{F}u\|_{L^{2}}^{2}.
  \end{align*}
  $H^{s}(\R)$ is a Hilbert space with inner product
  \begin{align*}
    \langle u,v \rangle_{H^{s}}&= \langle u,v \rangle_{L^{2}} + \langle |\eta|^{s/2}\mathcal{F} u, |\eta|^{s/2}\mathcal{F}v \rangle_{L^{2}} \\
&= \langle u,v \rangle_{L^{2}} + C_{s}\iint_{\R \times \R} \frac{(u(x)-u(y))\overline{(v(x)-v(y))}}{|x-y|^{1+2s}} dx dy .
  \end{align*}
\end{prop}

For $s>1, s \not \in \N$, the fractional Sobolev space is (recursively) defined by requiring that $u \in H^{s-1}$ and $\p_{x}u \in H^{s-1}$.
The definition via a kernel can be adapted to other and higher dimensional domains. We, in particular, are interested in the setting of the interval $[0,1]$.

\begin{prop}[Trace map; {\cite[Section 3]{hitch}}]
  Let $0<s<1$ and define $H^{s}([0,1])$ as the closure of $C^{\infty}([0,1])$ with respect to 
  \begin{align*}
    \iint_{[0,1]^{2}} \frac{|u(x)-u(y)|^{2}}{|x-y|^{1+2s}} dx dy + \|u\|_{L^{2}([0,1])}^{2}.
  \end{align*}
  Then $H^{s}([0,1])$ is a Hilbert space.
  Let further $s>1/2$, then $H^{s}$ embeds into $C^{0}$, in particular there exists a trace map and 
  \begin{align*}
    \left| u_{y=0,1} \right| \lesssim_{s} \|u\|_{H^{s}([0,1])}.
  \end{align*}
\end{prop}

A closely related space is given by the periodic fractional Sobolev space $H^{s}(\T)$.

\begin{prop}[{\cite{benyi2013sobolev}}]
   Let $0<s<1/2$, then for any $u \in C^{\infty}(\T)$,
  \begin{align*}
    \||\eta|^{s}\mathcal{F}u\|_{L^{2}}^{2}  \lesssim \iint_{\T \times [-\frac{1}{2},\frac{1}{2}]} \frac{|u(x+y)-u(y)|^{2}}{|x|^{1+2s}} dx dy \lesssim \||\eta|^{s}\mathcal{F}u\|_{L^{2}}^{2}.
  \end{align*}
  In particular, both the kernel and Fourier characterization define the same quasi-norm.
  Furthermore, 
  \begin{align*}
     \iint_{\T \times [-\frac{1}{2},\frac{1}{2}]} \frac{|u(x+y)-u(y)|^{2}}{|x|^{1+2s}} dx dy = \langle \mathcal{F}u, B_{n}|n|^{2s} \mathcal{F} u \rangle_{l^{2}},
  \end{align*}
  where $B_{n}$ satisfies 
  \begin{align*}
    1 \lesssim B_{n}:= |n|^{-2s} \int_{[-\frac{1}{2},\frac{1}{2}]} \frac{\sin^{2}(xn)}{4|x|^{1+2s}} dx \lesssim 1.
  \end{align*}
  The fractional Sobolev space $H^{s}(\T)$ is defined as the closure of $C^{\infty}(\T)$ with respect to
  \begin{align*}
    \|u\|_{L^{2}}^{2}+ \iint_{\T \times [-\frac{1}{2},\frac{1}{2}]} \frac{|u(x+y)-u(y)|^{2}}{|x|^{1+2s}} dx dy.
  \end{align*}
  $H^{s}(\T)$ is a Hilbert space, where the inner product can be chosen as either
  \begin{align*}
    \langle u,v \rangle_{H^{s}(\T)}:&=\langle u,v \rangle_{L^{2}} +  \langle \mathcal{F} u , B_{n}|n|^{2s} \mathcal{F}v \rangle_{l^{2}} \\
&=
 \langle u,v \rangle_{L^{2}} + \iint \frac{\overline{(u(x+y)-u(y))}(v(x+y)-v(y))}{|x|^{1+2s}} dx dy,
  \end{align*}
  or 
  \begin{align*}
    \langle u,v \rangle_{H^{s}(\T)}:=  \langle u,v \rangle_{L^{2}} + \langle \mathcal{F} u , |n|^{2s} \mathcal{F}v \rangle_{l^{2}}.
  \end{align*}
\end{prop}

From the kernel characterization, it can easily be seen that $H^{s}(\T) \subset H^{s}([0,1])$:
\begin{prop}
\label{prop:simpleinclusion}
  Let $0<s<1$, then any $u \in H^{s}(\T)$ is also in $H^{s}([0,1])$ and
  \begin{align*}
    \|u\|_{H^{s}([0,1])} \lesssim \|u\|_{H^{s}(\T)}.
  \end{align*}
\end{prop}

\begin{proof}[Proof of Proposition \ref{prop:simpleinclusion}]
The $L^{2}$ norms are equal, hence we only have to consider the quasi-norm in $H^{s}([0,1])$.
Introducing a change of variables $x \mapsto z+y$, we compute
\begin{align*}
 \|u\|_{\dot H^{s}([0,1])}^{2} &= \iint_{[0,1]^{2}} \frac{|u(x)-u(y)|^{2}}{|x-y|^{1+2s}} dx dy \\
&= \int_{[0,1]} \int_{[0,1]-y} \frac{|u(z+y)-u(y)|^{2}}{|z|^{1+2s}} dz dy \\
& \leq \int_{[0,1]} \int_{[-1,2]} \frac{|u(z+y)-u(y)|^{2}}{|z|^{1+2s}} dz dy \\
& \leq \|u\|_{\dot H^{s}(\T)}^{2} + C \|u\|_{L^{2}}^{2} \lesssim \|u\|_{H^{s}(\T)}^{2},
\end{align*}
where we used that 
\begin{align*}
  \sup_{|z|\geq \frac{1}{2}}\frac{1}{|z|^{1+2s}} \leq 2.
\end{align*}
\end{proof}

As a simplification, for the stability results of Section \ref{sec:stability-h32-} and Section \ref{sec:stability-h52-}, we restrict to fractional Sobolev spaces, $H^{s}(\T)$, in order to be able to use the Fourier characterization.
In this case,  we further require that the coefficient functions, $f,g$, corresponding to the shear flow, $U$, are not only sufficiently regular, e.g. $g \in W^{1,\infty}([0,1])$,
but can be periodically extended in a regular way, e.g. $g \in W^{1,\infty}(\T)$, in order to be able to apply the following Propositions \ref{prop:LipHs} and \ref{prop:ComHs}.

\begin{prop}[Multiplication with Lipschitz functions]
  \label{prop:LipHs}
  Let $g \in W^{1,\infty}(\T)$ be periodic and Lipschitz, then for any $s<1/2$ and any $u \in H^{s}(\T)$,
  also $gu \in H^{s}(\T)$ and 
  \begin{align*}
    \|ug\|_{H^{s}} \leq \|g\|_{W^{1,\infty}}\|u\|_{H^{s}}.
  \end{align*}
\end{prop}

\begin{prop}[Commutator Estimate]
\label{prop:ComHs}
  Let $g \in C^{0,1}(\T)$ with $g^{2}>c>0$ and let $0<s<1/2$. Then for any $u \in H^{s}(\T)$
  \begin{align*}
   \Re \langle u, g^{2}u \rangle_{H^{s}(\T)} \geq c\|u\|_{H^{s}(\T)}^{2} -C_{s}\|g^{2}\|_{\dot{C}^{0,1}} \|u\|_{L^{2}}^{2}.
  \end{align*}
\end{prop}

\begin{rem}
\label{rem:periodicityassumptions}
The periodicity assumption on $g$ drastically simplifies calculations, but can probably be relaxed.

It can be shown that the multiplication with the characteristic function of the positive half-line, $1_{[0,\infty)}$, is a bounded operator on $H^{s}(\R), s< \frac{1}{2}$(see \cite[page 208]{runst1996sobolev}).
Thus, one can probably allow for a jump discontinuity of the periodic extension of $g$ in Proposition \ref{prop:LipHs} and only require that $g \in W^{1,\infty}([0,1])$.

In the case of Proposition \ref{prop:ComHs}, we, however, use that the commutator
\begin{align*}
u \mapsto  [(-\Delta)^{\frac{s}{2}},g]u,
\end{align*}
where $(-\Delta)^{\frac{s}{2}}$ is defined as the Fourier multiplier
\begin{align*}
  u \mapsto \mathcal{F}^{-1}|\eta|^{s}\mathcal{F} u,
\end{align*}
is not only a bounded operator from $H^{s}$ to $L^{2}$, but gains regularity in the sense that it also is a bounded operator from $H^{s-\epsilon}$ to $L^{2}$ for some $\epsilon>0$.
As this is not the case for functions with jump discontinuities, the current proof can probably only be extended to functions $g$, for which the size of the jump discontinuity
\begin{align*}
  |g^{2}(1)-g^{2}(0)|
\end{align*}
is sufficiently small compared to $\min(g^{2})>0$, so that the possible loss due to the jump satisfies (by the improved version of Proposition \ref{prop:LipHs})
\begin{align*}
  |g^{2}(1)-g^{2}(0)| \|1_{[\frac{1}{2},1]}u\|_{H^{s}}^{2} \leq \frac{\min(g^{2})}{2}\|u\|_{H^{s}}^{2}
\end{align*}
and can hence be absorbed by
\begin{align*}
  \langle u, \min(g^{2})u \rangle_{H^{s}} = \min(g^{2})\|u\|_{H^{s}}^{2}.
\end{align*}

Removing the restriction on the size of the jump,
\begin{align*}
  |g^{2}(1)-g^{2}(0)|,
\end{align*}
 is probably possible, but would require considerable additional technical effort.
\end{rem}

\begin{proof}[Proof of Proposition \ref{prop:LipHs}]
  We remark that $gu \in L^{2}$ and that $\|gu\|_{L^{2}} \leq \|g\|_{W^{1,\infty}}\|u\|_{L^{2}}$ is well-known.
  For the $H^{s}$ seminorm we follow the standard proof via the kernel characterization (see \cite[page 21]{hitch}). 
  \begin{align*}
    &\quad \iint_{\T\times [-\frac{1}{2}, \frac{1}{2}]} \frac{|u(x+y)g(x+y)- u(y) g(y)|^{2}}{|x|^{1+2s}} dx dy \\
& \lesssim \iint_{\T \times [-\frac{1}{2}, \frac{1}{2}]} |g(x+y)|^{2}\frac{|u(x+y)- u(y)|^{2}}{|x|^{1+2s}} dx dy \\
& \quad +  \iint_{\T \times [-\frac{1}{2}, \frac{1}{2}]} |u(y)|^{2}\frac{|g(x+y)- g(y)|^{2}}{|x|^{1+2s}} dx dy .
  \end{align*}
    The first term can be easily controlled by $\|g\|_{L^{\infty}}^{2} \|u\|_{H^{s}}^{2}$.
  For the second term we use that $g \in W^{1,\infty}(\T)$ is Lipschitz and thus 
  \begin{align*}
    \frac{|g(x)-g(y)|^{2}}{|x-y|^{1+2s}} \leq \frac{1}{|x-y|^{2s-1}} \|g\|_{W^{1,\infty}}^{2}.
  \end{align*}
  Then,
  \begin{align*}
    \sup_{y \in \T} \int_{ [-\frac{1}{2}, \frac{1}{2}]} \frac{1}{|x-y|^{2s-1}} dx \leq \int_{-1}^{2} \frac{1}{|x|^{2s-1}} dx < \infty, 
  \end{align*}
  as $1-2s>-1$ for all $0<s<1$.
  The second term can thus be controlled in terms of $\|u\|_{L^{2}}^{2} \|g\|_{W^{1,\infty}}^{2}$.
\end{proof}

\begin{proof}[Proof of Proposition \ref{prop:ComHs}]
For the $L^{2}$ product there is nothing to show.

By the kernel characterization 
\begin{align*}
\Re \langle u, g^{2}u \rangle_{H^{s}(\T)} &= \Re \iint \frac{\overline{(u(x+y)-u(y))}(g^{2}(x+y)u(x+y)-g^{2}(y)u(y))}{|x|^{1+2s}} dx dy \\
& =  \iint g^{2}(x+y) \frac{|u(x+y)-u(y)|^{2}}{|x|^{1+2s}} dx dy \\
& \quad -  \Re \iint \frac{g^{2}(x+y)-g^{2}(y)}{|x|^{1+2s}} \overline{(u(x+y)-u(y))} u(y) dx dy.
\end{align*}
As $g^{2}$ is Lipschitz, the second term can thus be estimated by 
\begin{align*}
&\quad \|g^{2}\|_{W^{1,\infty}}\iint \frac{1}{|x|^{2s}}|u(x+y)-u(y)| |u(y)| dx dy \\
&\leq 2\|g^{2}\|_{W^{1,\infty}}\left\|\frac{1}{|x|^{2s}}\right\|_{L^{1}_{x}} \|u\|_{L^{2}}^{2} \leq C_{s}\|g^{2}\|_{W^{1,\infty}} \|u\|_{L^{2}}^{2},
\end{align*}
where we used that $2s<1$.
\end{proof}

\section{Boundary effects and sharp stability results}

In a previous article, \cite{Zill3}, we proven that the linearized 2D Euler equations in a finite periodic channel, $\T_{L}\times [0,1]$, are stable in $H^{m}_{x}H^{1}_{y}(\T_{L}\times [0,1])$ for general perturbations, but only stable in $H^{m}_{x}H^{2}_{y}(\T_{L}\times [0,1])$ under perturbations with zero Dirichlet boundary data, $\omega_{0}|_{y=0,1}=0$.
\\

In this section, we study the boundary effects and the associated singularity formation in more detail and show that the critical Sobolev spaces in $y$ are given by $H^{\frac{3}{2}}_{y}$ and $H^{\frac{5}{2}}_y$, respectively.
More precisely, we show that stability in $H^{m}_{x}H^{s}_{y}(\T_{L}\times [0,1]), s>\frac{3}{2}$ can not hold for general perturbations due the development of (logarithmic) singularities at the boundary.
On the other hand, we prove stability in $H^{m}_{x}H^{s}_{y}(\T_{L}\times \T)$ for any $s<\frac{3}{2}$, where for technical reasons we consider periodic fractional Sobolev spaces, $H^{s}(\T)$, instead of $H^{s}([0,1])$.
In particular, stability in $H^{s},s>1$, allows us to prove damping with an integrable rate and thus obtain a quantitative scattering results for initial perturbations without zero Dirichlet data, which has not been possible with the $H^{1}$ stability results of \cite{Zill3}. 
 
Restricting to perturbations with zero Dirichlet boundary data,  $\omega_{0}|_{y=0,1}=0$, we similarly show that the critical space is given by $H^{\frac{5}{2}}$ and prove stability and instability for $H^{m}_{x}H^{s}_{y}(\T_{L}\times \T),s<\frac{5}{2},$ and $H^{m}_{x}H^{s}_{y}(\T_{L}\times[0,1]),s>\frac{5}{2}$, respectively.
As we discuss in Section \ref{cha:anoth-sect-cons}, these improvements allow us to study consistency of the nonlinear problem in the finite periodic channel, where the singularity formation at the boundary and the resulting regularity restrictions have a large effect on possible nonlinear damping results.

\subsection{Stability in $H^{3/2-}$ and boundary perturbations}
\label{sec:stability-h32-}

In \cite{Zill3}, we established stability of the linearized Euler equations, \eqref{eq:LE}, in a finite periodic channel, $\T_{L}\times [0,1]$, in $H^{m}_{x}H^{1}_{y}$, for general initial data.
The damping result, Theorem \ref{thm:lin-zeng}, hence provides decay of the perturbations to the velocity field with rate $t^{-1}$, i.e.
\begin{align}
  \label{eq:6}
  \begin{split}
  \|v- \langle v \rangle_{x}\|_{L^{2}_{x,y}(\T_{L}\times [0,1])}&= \mathcal{O}(t^{-1}),\\
   \|v_{2}\|_{L^{2}_{x,y}(\T_{L}\times [0,1])} &= \mathcal{O}(t^{-1}).
  \end{split}
\end{align}
As this is almost sufficient to establish scattering, a natural question to ask is how far this can be improved,
that is for which values of $s$, with $s>1$, stability in $H^{m}_{x}H^{s}_{y}$ still holds.
\\

As the main result of this section, we show that the critical Sobolev exponent in $y$ is given by $s=\frac{3}{2}$.
More precisely, in the Corollaries \ref{cor:not32} and \ref{cor:phiboundaryh2}, we show that for perturbations $\omega_{0}$ with non-vanishing Dirichlet data, $\omega_{0}|_{y=0,1}$, $\p_{y}W$ asymptotically develops (logarithmic) singularities at the boundary and that hence stability in $H^{m}_{x}H^{s}_{y}(\T_{L}\times [0,1]), s>\frac{3}{2}$, and $H^{m}_{x}H^{2}_{y}(\T_{L}\times [0,1])$ can not hold, unless one restricts to perturbations $\omega_{0}$ such that $\omega_{0}|_{y=0,1}=0$.
This singularity formation is further analyzed in Section \ref{sec:boundary-layers}, where we also study the behavior close to the boundary and the heuristic implications for stability in $L^{p}$ spaces.
As we discuss in Section \ref{cha:anoth-sect-cons}, these instability results have strong implications for the problem of nonlinear inviscid damping in a finite channel. 

As a complementary result to the singularity formation, Theorem \ref{thm:H32} establishes stability in the \emph{periodic} fractional Sobolev spaces $H^{m}_{x}H^{s}_{y}(\T_{L}\times\T), s<3/2$.
In particular, we thus obtain inviscid damping with an integrable (but subquadratic) rate and hence scattering for initial perturbations without zero Dirichlet data, which has not been possible with the $H^{1}$ stability results of \cite{Zill3}, but only with the $H^{2}$ stability results for perturbations with vanishing Dirichlet data.
\\

We recall that the linearized 2D Euler equations in a finite periodic channel, $\T_{L}\times [0,1]$, are given by:
 \begin{align}
   \label{eq:Eulernochmal}
     \begin{split}
   \dt W &= \frac{if(y)}{k}\Phi, \\
   (-1+(g(y)(\frac{\p_{y}}{k} -it))^{2}) \Phi &= W,\\
   \Phi|_{y=0,1}&=0, \\
   (t,k,y) &\in  \R \times L(\Z\setminus \{0\}) \times [0,1].
       \end{split}
 \end{align}
Furthermore, as noted in Section \ref{sec:linearized-2d-euler}, the equations \eqref{eq:Eulernochmal} decouple with respect to $k$.
Hence, for the remainder of this section, we consider $k$ as a given parameter and consider the stability of 
\begin{align*}
W(t)=W(t,k,\cdot)\in H^{s}([0,1]).
\end{align*}
Results for $H^{m}_{x}H^{s}_{y}(\T_{L}\times [0,1]), m \in \N_{0},$ can then be obtained by summing over $k$.
\\

Considering the evolution of $\p_{y}W$:
\begin{align}
\label{eq:pyW}
  \begin{split}
  \dt \p_{y}W &= \frac{if}{k} \p_{y}\Phi + \frac{if'}{k}\Phi, \\
  (-1+(g(\frac{\p_{y}}{k}-it))^{2})\Phi^{(1)} &= \p_{y}W + [(g(\p_{y}-it))^{2}, \p_{y}] \Phi, \\
  \Phi^{(1)}_{y=0,\pi} &= 0 , \\
  H^{(1)}&= \p_{y}\Phi -\Phi^{(1)}, \\
   (t,k,y) &\in  \R \times L(\Z\setminus \{0\}) \times [0,1],
  \end{split}
\end{align}
at the boundary, $y \in \{0,1\}$, we prove that Sobolev stability can not hold for $s>\frac{3}{2}$, unless one restricts to perturbations $\omega_{0}$ with $\omega_{0}|_{y=0,1}\equiv 0$.
In that case, as we show in Section \ref{sec:stability-h52-}, an instability develops for $s>\frac{5}{2}$. 

Using a similar approach as in Section \ref{sec:const-coeff-model}, the following lemma provides a characterization of $\p_{y}\Phi|_{y=0,1}$ and describes the asymptotic behavior.

\begin{lem}
\label{lem:boundary32}
  Let $W$ be a solution of the linearized Euler equations, \eqref{eq:Eulernochmal}, and suppose that $g \in W^{2,\infty}([0,1])$ satisfies $g^{2}>c>0$. Then, 
  \begin{align}
    \label{eq:pyPhi}
    \begin{split}
    \p_{y}\Phi |_{y=0}&= \frac{k}{g^{2}(0)} \langle W, u_{1} \rangle,\\
    \p_{y}\Phi |_{y=1}&= \frac{k}{g^{2}(1)} \langle W, u_{2} \rangle,\\
    \end{split}
  \end{align}
  where 
  \begin{align*}
    u_{1}(t,y)&= e^{ikty} u_{1}(0,y) , \\
    u_{2}(t,y)&= e^{ikt(y-1)}u_{2}(0,y),
  \end{align*}
  and $u_{j}(0,y)$ are solutions of 
  \begin{align*}
    (-k^{2}+(g\p_{y})^{2})u &=0, \\
    y &\in [0,1],
  \end{align*}
  with boundary values 
  \begin{align}
    \label{eq:7}
    \begin{split}
      u_{1}(0,0)=u_{2}(0,1)&=0, \\
      u_{2}(0,1)=u_{2}(0,0)&=0.
    \end{split}
  \end{align}
\\

  Let $s>0$ and suppose that
  \begin{align*}
    \|\p_{y} W(t)\|_{H^{s}}<C<\infty
  \end{align*}
  for all time, then, as $t \rightarrow \infty$, 
  \begin{align*}
    \langle W, u_{1} \rangle  = \frac{1}{ikt} \omega_{0}|_{y=0} + \mathcal{O}(t^{-1-s}), \\
    \langle W, u_{2} \rangle  = \frac{1}{ikt} \omega_{0}|_{y=1} + \mathcal{O}(t^{-1-s}).
  \end{align*}
\end{lem}

As a corollary, we see that stability in $s>3/2$ can in general not hold.
\begin{cor}
\label{cor:not32}
Let $W$ be a solution of the linearized Euler equations, \eqref{eq:Eulernochmal}, and suppose that $f,g \in W^{2,\infty}([0,1])$ and that $g$ satisfies $g^{2}>c>0$. 
  Let $s>1$ and suppose that 
  \begin{align*}
    \|\p_{y} W\|_{H^{s}([0,1])}<C<\infty .
  \end{align*}
  Suppose further that $f \omega_{0}|_{y=0,1}$ is non-trivial. Then  
  \begin{align*}
    \|\p_{y}W(t)\|_{L^{\infty}([0,1])} \gtrsim \log |t|
  \end{align*}
  as $t \rightarrow \pm \infty$.

  As a consequence, for perturbations such that $f \omega_{0}|_{y=0,1}$ is non-trivial, for any $s>\frac{3}{2}$, necessarily
  \begin{align*}
    \sup_{t>0} \|W(t)\|_{H^{s}([0,1])} = \infty .
  \end{align*}
\end{cor}

\begin{proof}[Proof of Corollary \ref{cor:not32}]
Restricting \eqref{eq:pyW} to the boundary, we obtain 
\begin{align*}
  \dt \p_{y}W|_{y=0,1}= \frac{if}{k} \p_{y}\Phi|_{y=0,1},
\end{align*}
where we used that $\Phi|_{y=0,1}=0$.

By Lemma \ref{lem:boundary32}, under the assumptions of the corollary, thus 
\begin{align*}
  \dt \p_{y}W|_{y=0,1} =\frac{1}{t} \left. \frac{if}{k}\omega_{0} \frac{k}{g^{2}}\right|_{y=0,1} + \mathcal{O}(t^{-1-s}).
\end{align*}
Integrating this equality and using that 
\begin{align*}
  \left. \frac{if}{k}\omega_{0} \frac{k}{g^{2}}\right|_{y=0,1}
\end{align*}
is independent of $t$ and non-trivial, 
\begin{align*}
  |\p_{y}W|_{y=0,1}(t)| \gtrsim \int^{t} \frac{1}{\tau} - \mathcal{O}(\tau^{-1-s}) d\tau  \gtrsim \log |t|,
\end{align*}
which provides the lower bound on $\|\p_{y}W\|_{L^{\infty}}$ and hence the first result.
\\

The second result is proven by contradiction.
Let thus $s>3/2$ be given and suppose to the contrary that 
\begin{align*}
  \|W(t)\|_{H^{s}}< C<\infty,
\end{align*}
uniformly in time.
Then, by the trace map and the first result, 
  \begin{align*}
   \log(t) \lesssim \|\p_{y}W\|_{L^{\infty}} \lesssim_{s} \|W(t)\|_{H^{s}}<C,
  \end{align*}
which is a contradiction as $t \rightarrow \infty$.
\end{proof}

\begin{proof}[Proof of Lemma \ref{lem:boundary32}]
We note that, by construction, $u_{1}(t,y)$ and $u_{2}(t,y)$ are solutions of 
\begin{align*}
 (-1+(g(\frac{\p_{y}}{k}-it))^{2})u_{j}=0 
\end{align*}
with boundary values
  \begin{align}
    \begin{split}
      u_{1}(t,0)=u_{2}(t,1)&=0, \\
      u_{2}(t,1)=u_{2}(t,0)&=0,
    \end{split}
  \end{align}
for all times $t$.
Hence, integrating by parts, we obtain
\begin{align*}
  \langle W, u_{j} \rangle &= \langle (-1+(g(\frac{\p_{y}}{k}-it))^{2})\Phi, u_{j} \rangle \\ &= \overline{u}_{j} \frac{g^{2}}{k}(\frac{\p_{y}}{k}-it)\Phi|_{y=0}^{1} - \Phi \frac{g^{2}}{k}(\frac{\p_{y}}{k}-it)u_{j}|_{y=0}^{1} + \langle \Phi,  (-1+(g(\frac{\p_{y}}{k}-it))^{2})u_{j} \rangle\\
&= \overline{u}_{j} \frac{g^{2}}{k} \p_{y}\Phi|_{y=0}^{1},
\end{align*}
where we used that $\Phi|_{y=0,1}=0$. Using the boundary values of $u_{j}$ then yields \eqref{eq:pyPhi}.

Integrating 
\begin{align*}
  u_{1}(t,y) = e^{ikty}u_{1}(0,y)= u_{1}(0,1)\p_{y}\frac{e^{ikty}}{ikt}
\end{align*}
by parts, we obtain a boundary term 
\begin{align*}
  \frac{1}{ikt} W u_{1}|_{y=0,1} = - \frac{1}{ikt} W|_{y=0} = - \frac{1}{ikt} \omega_{0}|_{y=0},
\end{align*}
as well as a bulk term
\begin{align*}
 \frac{1}{ikt} \langle e^{ikty}, \p_{y}(Wu_{1}(0,y)) \rangle  = \frac{1}{ikt}\langle e^{ikty}u_{1}, \p_{y}W \rangle + \frac{1}{ikt}\langle e^{ikty}\p_{y}u_{1}, W \rangle .
\end{align*}
The boundary term is already of the desired form.

The second term of the bulk contribution can be integrated by parts once more and thus yields a quadratically decaying contribution.
It thus remains to estimate the first term,
\begin{align*}
  \frac{1}{ikt}\langle e^{ikty}u_{1}, \p_{y}W \rangle.
\end{align*}
There, we use duality and estimate 
\begin{align*}
  \langle e^{ikty}u_{1}, \p_{y}W \rangle_{L^{2}} \leq \|e^{ikty}u_{1}\|_{H^{-s}} \|\p_{y}W\|_{H^{s}} = \mathcal{O}(t^{-s})\|\p_{y}W\|_{H^{s}}.
\end{align*}
\end{proof}
As a consequence we show that for stability in $H^{2}$ it is necessary to restrict to perturbations with vanishing Dirichlet boundary data, $\omega_{0}|_{y=0,1}=0$:
\begin{cor}
\label{cor:phiboundaryh2}
  Let $\omega_{0} \in H^{2},f,g$ satisfy the assumptions of Lemma \ref{lem:boundary32} and suppose that $f\omega_{0}|_{y=0,1}$ is non-trivial.
Let further $W(t)$ be the solution of the linearized Euler equations, \eqref{eq:Eulernochmal}.
  Then, 
  \begin{align*}
    \sup_{t}\|W(t)\|_{H^{2}([0,1])} = \infty.
  \end{align*}
\end{cor}

\begin{proof}[Proof of Corollary \ref{cor:phiboundaryh2}]
  We follow the same strategy as in the proof of Corollary \ref{cor:not32}.
  Thus, assume to the contrary that $\|W(t)\|_{H^{2}}$ is bounded uniformly in time.
  Then, for example at $y=0$, 
  \begin{align*}
     \langle W, e^{ity} u_{1} \rangle_{L^{2}}&= \frac{1}{ikt} W|_{y=0}- \frac{1}{ikt}\langle e^{ikty}, \p_{y}(Wu_{1}) \rangle_{L^{2}}, \\
    &= \frac{1}{ikt} W|_{y=0} + \frac{1}{k^{2}t^{2}} \p_{y}(Wu_{1})|_{y=0}^{1}- \frac{1}{k^{2}t^{2}} \langle e^{ikty}, \p_{y}^{2}(Wu_{1}) \rangle_{L^{2}}.
  \end{align*}
  Both the last $L^{2}$ product and the trace of $W$ and $\p_{y}W$ can be controlled by $\|W\|_{H^{2}([0,1])}$. Thus,
  \begin{align*}
   \p_{y}\Phi|_{y=0} = \frac{k}{g^{2}(0)}\langle W, e^{ity} u_{1} \rangle_{L^{2}} = \frac{1}{itg^{2}(0)} \omega_{0}|_{y=0} + \mathcal{O}(t^{-2})\|W\|_{H^{2}([0,1])},
  \end{align*}
  where we used Lemma \ref{lem:boundary32}.

  Integrating 
  \begin{align*}
    \dt \p_{y}W|_{y=0,1} = \frac{if}{k}\p_{y}\Phi|_{y=0,1},
  \end{align*}
  in $t$, thus yields that $\p_{y}W|_{y=0,1}$ blows up logarithmically as $t \rightarrow \infty$. On the other hand, the $L^{\infty}$ norm of $\p_{y}W$ is controlled by the $H^{2}$ norm via the Sobolev embedding theorem, which yields a contradiction.
\end{proof}

We have thus seen that, in general, for the purposes of stability results $s$ can not be larger than $3/2$.
The main result of this section is that this condition is sharp in the sense that stability in $H^{s}$ holds for all $s<3/2$.
More precisely, instead of $H^{s}([0,1])$, we consider \emph{periodic} spaces, i.e.
\begin{align*}
  W(t,k,\cdot) \in H^{s-1}(\T), \p_{y}W(t,k,\cdot) \in H^{s-1}(\T),
\end{align*}
where $\T=[0,1]/\sim$ is the torus of unit period.
As discussed in Section \ref{sec:fract-sobol-space}, this allows us to use both a Fourier characterization and a kernel characterization. 

\begin{thm}
\label{thm:H32}
  Let $0<s<1/2$, $\omega_{0} \in H^{1}([0,1])$ and $\omega_{0},\p_{y}\omega_{0} \in H^{s}(\T)$. Suppose further that $f,g \in W^{2,\infty}(\T)$, that there exists $c>0$ such that 
  \begin{align*}
    0 < c< g< c^{-1}< \infty,
  \end{align*}
 and that 
  \begin{align*}
    \|f\|_{W^{2,\infty}(\T)} L 
  \end{align*}
  is sufficiently small.
  Then the solution, $W$, of the linearized Euler equations, \eqref{eq:Eulernochmal}, satisfies 
  \begin{align*}
    \|\p_{y}W(t)\|_{H^{s}(\T)} \lesssim \|\omega_{0}\|_{H^{s}(\T)}+\|\p_{y}\omega_{0}\|_{H^{s}(\T)} ,
  \end{align*}
 uniformly in time.
\end{thm}

\begin{rem}
\label{rem:H32assumptionsperiod}
  The assumptions on $f$ and $g$ are chosen such that we can apply  Proposition \ref{prop:LipHs} to the functions $f$, $g$ and their derivatives $f'$ and $g'$.
Furthermore, we require 
\begin{align*}
  g^{2}=U'(U^{-1}(\cdot))^{2}
\end{align*}
to be such that we can apply Proposition \ref{prop:ComHs}.

As discussed in Remark \ref{rem:periodicityassumptions}, these assumptions can probably be relaxed to requiring that 
\begin{align*}
  f,g \in W^{3,\infty}([0,1]),
\end{align*}
and that 
\begin{align*}
  |g^{2}(1)-g^{2}(0)| = |(U'(b))^{2}-(U'(a))^{2}|
\end{align*}
is sufficiently small compared to 
\begin{align*}
  \min(g^{2})= \min((U')^{2}) >0.
\end{align*}
\end{rem}

\begin{proof}[Proof of Theorem \ref{thm:H32}]

In our proof, we split $\p_{y}\Phi$ into a solution with zero Dirichlet boundary conditions and a correction term in the form of a homogeneous solution:
\begin{align*}
  \dt \p_{y} W &= ikf \Phi^{(1)} + ikf' \Phi + ikf H^{(1)}, \\
  (-k^{2}+ (g(\p_{y}-ikt))^{2})\Phi^{(1)}&= \p_{y}W + [(g(\p_{y}-ikt))^{2}, \p_{y}] \Phi, \\
  \Phi^{(1)}|_{y=0,1} &=  0, 
\end{align*}
where $H^{(1)}$ is given by
\begin{align*}
  (-k^{2}+ (g(\p_{y}-ikt))^{2})H^{(1)}&= 0 , \\ 
  H^{(1)} &= H^{(1)}|_{y=0} e^{ikty}u_{1} + H^{(1)}|_{y=1} e^{ikt(y-1)}u_{2}, \\ 
  H^{(1)}|_{y=0} &= \p_{y}\Phi |_{y=0} = \frac{1}{g^{2}} \langle W, e^{ikty} u_{1} \rangle ,\\
  H^{(1)}|_{y=1} &= \p_{y}\Phi |_{y=1} = \frac{1}{g^{2}} \langle W, e^{ikt(y-1)} u_{2} \rangle .
\end{align*}

Considering a decreasing weight $A$ and computing 
\begin{align*}
  \dt (\langle W, A W \rangle_{H^{s}} + \langle \p_{y}W, A \p_{y} W \rangle_{H^{s}}) =: \dt I(t),
\end{align*}
we thus have to control 
\begin{align*}
\tag{elliptic}
  & \langle \frac{if}{k} \Phi, A W \rangle_{H^{s}} + \langle \frac{if}{k} \Phi^{(1)},A \p_{y}W \rangle_{H^{s}}  + \langle \frac{if'}{k} \Phi,A \p_{y}W \rangle_{H^{s}}  \\
\tag{boundary}
+&  \langle \frac{if}{k} H^{(1)}, A \p_{y}W \rangle_{H^{s}}
\end{align*}
in terms of 
\begin{align*}
 \frac{C}{k} |\langle W, \dot A W \rangle_{H^{s}} + \langle \p_{y}W, \dot A \p_{y} W \rangle_{H^{s}} | 
\end{align*}

Assuming this control and requiring $k$ to be sufficiently large such that $\frac{C}{k} \ll 1$, this then yields that $I(t)$ is non-increasing. In particular, 
\begin{align*}
  \|W\|_{H^{s}}^{2} + \|\p_{y}W\|_{H^{s}}^{2} \lesssim I(t) \leq I(0) \lesssim \|\omega_{0}\|_{H^{s}}^{2} + \|\p_{y} \omega_{0}\|_{H^{s}}^{2}.
\end{align*}
It remains to prove the elliptic and boundary control in the following subsections.
\end{proof}

\subsection{Boundary corrections}
\label{sec:boundary-corrections}
The control of the boundary term in the proof of Theorem \ref{thm:H32} is provided by the following theorem.
\begin{thm}
\label{thm:H32boundary}
  Let $0<s<1/2$ and let $W,f,g$ as in Theorem \ref{thm:H32}.
  Let further $A$ be a diagonal operator comparable to the identity, i.e.
  \begin{align*}
    A: e^{iny} \mapsto A_{n} e^{iny},
  \end{align*}
  with 
  \begin{align*}
    1 \lesssim A_{n} \lesssim 1,
  \end{align*}
  uniformly in $n$.

 Then,
  \begin{align*}
    | \langle A \p_{y}W, if H^{(1)} \rangle_{H^{s}}| \lesssim \sum_{n} c_{n}(t) <n>^{2s} |(\p_{y}W)_{n}|^{2},
  \end{align*}
  for a family $c_{n} \in L^{1}_{t}$, with $\|c_{n}\|_{L^{1}_{t}}$  bounded uniformly in $n$.
\end{thm}

\begin{proof}[Proof of Theorem \ref{thm:H32boundary}]
$H^{(1)}$ is explicitly given by
\begin{align*}
  H^{(1)}= \p_{y}\Phi|_{y=0} e^{ikty}u_{1} + \p_{y}\Phi|_{y=1}e^{ikty(y-1)}u_{2}.
\end{align*}
We hence have to estimate
\begin{align}
\label{eq:H1fracboundary}
  \langle A \p_{y}W, if H^{(1)} \rangle_{H^{s}} = \p_{y}\Phi|_{y=0} \langle A \p_{y}W, if u_{1} \rangle_{H^{s}}
+ \p_{y}\Phi|_{y=1} \langle A \p_{y}W, if u_{2} \rangle_{H^{s}}.
\end{align}

By Lemma \ref{lem:boundary32}
\begin{align*}
  \p_{y}\Phi|_{y=0}&= \frac{k}{g^{2}(0)} \langle W, e^{ikty}u_{1} \rangle \\
&=  \frac{k}{g^{2}(0)} \left( \frac{1}{ikt} \omega_{0}|_{y=0} +  \frac{1}{ikt} \langle e^{ikty}, \p_{y} W u_{1} \rangle \right) , \\
  \p_{y}\Phi|_{y=1}&= \frac{k}{g^{2}(1)} \langle W, e^{ikt(y-1)}u_{2} \rangle \\
&=\frac{k}{g^{2}(1)} \left( \frac{1}{ikt} \omega_{0}|_{y=1} +  \frac{1}{ikt} \langle e^{ikt(y-1)}, \p_{y} W u_{2} \rangle \right) .
\end{align*}

Let us for the moment concentrate on the terms not involving $\omega_{0}$. 
Using the control of $g$ and $\frac{1}{g}$, in order to estimate \eqref{eq:H1fracboundary}, we hence have to estimate 
  \begin{align}
    | \langle A \p_{y}W, \frac{if}{k} e^{ikty} u_{1} \rangle_{H^{s}} \frac{1}{t} \langle \p_{y}W, e^{ikty}u_{1} \rangle_{L^{2}} |
  \end{align}
Expanding this in a basis, using that $1\lesssim A_{n} \lesssim 1$, $f \in W^{1,\infty}$ and denoting
\begin{align*}
   b_{n}:=|(\p_{y}W)_{n}|,
\end{align*}
it suffices to consider
\begin{align}
\label{eq:thissplit}
  \frac{1}{t} \left(\sum_{n} b_{n} \frac{<n>^{2s}}{<n-kt>} \right) \left(\sum_{n} \frac{b_{n}}{<n-kt>}\right).
\end{align}
Considering the decay of the coefficients in $n$ and taking into account that we only control $b_{n}<n>^{s} \in l^{2}$, we need that
\begin{align*}
  \frac{<n>^{s}}{<n-kt>} \in l^{2},
\end{align*}
which is the case iff $s<1/2$.

As $0<s<1/2$, we may choose $0<\lambda<1$ such that $s-\lambda < -1/2$ and split 
\begin{align*}
  & \quad \sum_{n}b_{n} \frac{<n>^{s}}{<n-kt>^{1-\lambda}} \frac{<n>^{s}}{<n-kt>^{\lambda}} \\
&\leq \left(\sum b_{n}^{2} \frac{<n>^{2s}}{<n-kt>^{2(1-\lambda)}}\right)^{1/2} \left\|\frac{<n>^{s}}{<n-kt>^{\lambda}} \right\|_{l^{2}}.
\end{align*}
Spliting the second factor in \eqref{eq:thissplit} in the same way, it suffices to show that 
\begin{align*}
  c_{n}(t):= \frac{1}{t} \frac{1}{<n-kt>^{2(1-\lambda)}} \left\|\frac{<m>^{s}}{<m-kt>^{\lambda}} \right\|_{l^{2}_{m}} \left\|\frac{1}{<m>^{s}<m-kt>^{\lambda}} \right\|_{l^{2}_{m}}
\end{align*}
is in $L^{1}_{t}$ with $\|c_{n}\|_{L^{1}_{t}}$ bounded uniformly in $n$.
Estimating $<n>^{s} \lesssim <n-kt>^{s} + <kt>^{s}$, it suffices to show that 
\begin{align*}
  <kt>^{s} \left\|\frac{1}{<n>^{s}<n-kt>^{\lambda}} \right\|_{l^{2}} \lesssim 1.
\end{align*}
As $s-\lambda <-1/2$, there exists a $\delta>0$ such that $\lambda = 1/2+\delta +s$. We thus estimate 
\begin{align*}
  \left\|\frac{1}{<n>^{s}<n-kt>^{\lambda}}\right\|_{l^{2}} \leq \|\frac{1}{<n>^{s}<n-kt>^{s}}\|_{l^{\infty}} \|\frac{1}{<n-kt>^{1/2+\delta}}\|_{l^{2}}.
\end{align*}
Hence, 
\begin{align*}
  c_{n}(t) \lesssim  \frac{1}{t} \frac{1}{<n-kt>^{2(1-\lambda)}} \in L^{1}_{t} .
\end{align*}
\\

It remains to discuss
\begin{align*}
  \frac{1}{ikt} \omega_{0}|_{y=0,1}.
\end{align*}
As the trace of $\omega_{0}$ is controlled by its initial $H^{1}$ norm, we consider $\omega_{0}|_{y=0,1}$ as constants of size $1$ in the following.
Hence, we have to estimate
\begin{align*}
  \left|\langle A \p_{y}W, \frac{if}{k}e^{ity}u_{1} \rangle \frac{1}{kt}\right|.
\end{align*}
Splitting 
\begin{align*}
 \left|\frac{1}{kt}\right|= \left|\frac{1}{kt}\right|^{\gamma} \left|\frac{1}{kt}\right|^{1-\gamma},
\end{align*}
with $1/2<\gamma<1/2+\epsilon$ and using Young's inequality, we thus obtain 
\begin{align*}
  |\langle A \p_{y}W, \frac{if}{k}e^{ity}u_{1} \rangle \frac{1}{kt}| \lesssim 
  <kt>^{-2\gamma}+
  \left|\frac{1}{kt}\right|^{2(1-\gamma)} |\langle A \p_{y}W, \frac{if}{k}e^{ity}u_{1} \rangle|^{2}.
\end{align*}
Here, the first term is an integrable contribution.
Following the same strategy as above, the second term can be controlled by 
\begin{align*}
  \sum_{n} b_{n}^{2}\frac{<n>^{2s}}{<n-kt>^{2(1-\lambda)}} \frac{<kt>^{2s}}{<kt>^{2(1-\gamma)}}
\end{align*}
Choosing $\gamma, \lambda$ such that 
\begin{align*}
  s-(1-\lambda)-(1-\gamma) < -1/2
\end{align*}
and modifying $c_{n}(t)$ to also include
\begin{align*}
 \frac{<kt>^{2s}}{<kt>^{2(1-\gamma)}<n-kt>^{2(1-\lambda)}} \in L^{1}_{t},
\end{align*}
then proves the result. Such a choice is possible as $s<1/2$ is given and we can choose $(1-\lambda) < 1/2$ and $(1-\gamma)<1/2$ arbitrarily close to $1/2$.
\end{proof}

\subsection{Elliptic control}
\label{sec:elliptic-control}

In this section, our main goal is to prove the following theorem, which controls the elliptic contributions in the proof of Theorem \ref{thm:H32}.
Here, the main steps of the proof of Theorem \ref{thm:H32ellipticcontrol} are formulated as lemmata and propositions and conclude with Lemma \ref{lem:lastH32lemma}.
\begin{thm}
\label{thm:H32ellipticcontrol}
  Let $0<s<1/2$ and let $A,f,g,W$ as in Theorem \ref{thm:H32boundary}. Then
  \begin{align*}
    | \langle A \p_{y}W, if \Phi^{(1)} + if' \Phi \rangle_{H^{s}}| \lesssim \sum_{n} c_{n}(t) <n>^{2s} (|(\p_{y}W)_{n}|^{2} + |W_{n}|^{2}),
  \end{align*}
  for a family $c_{n} \in L^{1}_{t}$, where $\|c_{n}\|_{L^{1}_{t}}$ is bounded uniformly in $n$.  
\end{thm}

When working with non-fractional Sobolev spaces, in \cite{Zill3}, this estimate reduced to an elliptic regularity theorem of the form 
\begin{align*}
\|\Phi\|_{\tilde{H}^{1}} \lesssim \|W\|_{\tilde{H}^{-1}},
\end{align*}
where 
\begin{align*}
  \|\Phi\|_{\tilde{H}^{1}}^{2} = \|\Phi\|_{L^{2}}^{2}+ \|(\frac{\p_{y}}{k}-it)\Phi\|_{L^{2}}^{2}
\end{align*}
and $\tilde{H}^{-1}$ was constructed by duality.

Similarly, we show that the proof of Theorem \ref{thm:H32ellipticcontrol} reduces to estimating 
\begin{align*}
   \|\Phi\|_{H^{s}}^{2}+ \|(\frac{\p_{y}}{k}-it)\Phi\|_{H^{s}}^{2} 
+ \|\Phi^{(1)}\|_{H^{s}}^{2}+ \|(\frac{\p_{y}}{k}-it)\Phi^{(1)}\|_{H^{s}}^{2} .
\end{align*}

\begin{lem}
\label{lem:easy}
 Let $0<s<1/2$ and let $A,f,g,W$ be as in Theorem \ref{thm:H32ellipticcontrol}. Then 
 \begin{align*}
   &\quad \langle A \p_{y}W, if \Phi^{(1)} + if' \Phi \rangle_{H^{s}} \\
&\lesssim \left( \sum \frac{<n>^{2s}|(\p_{y}W)_{n}|^{2}}{<n-kt>^{2}}\right)^{1/2} (\|f' \Phi\|_{H^{s}} +\|f'' \Phi\|_{H^{s} } +  \|f' (\p_{y}-ikt) \Phi\|_{H^{s}} \\
& \quad+ \|f \Phi^{(1)}\|_{H^{s}} +\|f' \Phi^{(1)}\|_{H^{s} } +  \|f' (\p_{y}-ikt) \Phi^{(1)}\|_{H^{s}}).
 \end{align*}
\end{lem}

\begin{proof}[Proof of Lemma \ref{lem:easy}]
  Denote
  \begin{align*}
    R:= if \Phi^{(1)} + if' \Phi .
  \end{align*}
  Then,
  \begin{align*}
    \langle A \p_{y}W , R \rangle _{H^{s}} = \sum_{n} a_{n}(\p_{y}W)_{n} <n>^{2s} \langle e^{iny}, R \rangle .
  \end{align*}
Multiplying by a factor 
\begin{align*}
1=\frac{1+i(n/k-t)}{1+i(n/k-t)},
\end{align*}
we estimate 
\begin{align*}
  &\quad \sum_{n} \left( A_{n}(\p_{y}W)_{n} \frac{<n>^{s}}{1+i(n/k-t)} \right) \left(<n>^{s}(1+i(n/k-t))\langle e^{iny}, R \rangle \right)\\
&\leq \left\|A_{n}(\p_{y}W)_{n} \frac{<n>^{s}}{1+i(n/k-t)}  \right\|_{l^{2}_{n}}
\left\| <n>^{s}(1+i(n/k-t))\langle e^{iny}, R \rangle \right\|_{l^{2}_{n}}.
\end{align*}
We, in particular, note that 
\begin{align*}
  \frac{1}{|1+i(n/k-t)|^{2}} \in L^{1}_{t}.
\end{align*}

Thus, it suffices to control 
\begin{align}
\label{eq:Rnkt}
  \sum_{n} <n>^{2s} | (1+i(n/k-t))\langle e^{iny}, R \rangle |^{2}.
\end{align}
As
\begin{align*}
 in e^{iny} = \p_{y} e^{iny},
\end{align*}
and as $R$ has zero boundary values, integrating by parts yields
\begin{align*}
  (1+i(n/k-kt))\langle e^{iny}, R \rangle = \langle e^{iny}, R \rangle + \langle e^{iny}, (\frac{\p_{y}}{k}-it) R \rangle .
\end{align*}
By the triangle inequality and Young's inequality, one thus obtains an estimate of \eqref{eq:Rnkt} by 
\begin{align*}
  \|R\|_{H^{s}}^{2} + \|(\frac{\p_{y}}{k}-it) R\|_{H^{s}}^{2}.
\end{align*}
Computing $(\frac{\p_{y}}{k}-it) R$ by the product rule and using the triangle inequality then concludes the proof. 
\end{proof}

By Proposition \ref{prop:LipHs} of Section \ref{sec:fract-sobol-space}, for $f,g$ sufficiently regular, it hence suffices to estimate 
\begin{align*}
  \|\Phi\|_{H^{s}} + \|(\frac{\p_{y}}{k}-it) \Phi\|_{H^{s}}
+ \|\Phi^{(1)}\|_{H^{s}} + \|(\frac{\p_{y}}{k}-it) \Phi^{(1)}\|_{H^{s}}.
\end{align*}
As the estimates for $\Phi$ and $\Phi^{(1)}$ are very similar, to simplify notation and as we will later on also derive such an estimate for $\Phi^{(2)}$, we in the following consider a general problem:

Let $\psi$ solve
\begin{align*}
\tag{ELL}
\label{eq:ELL}
  \begin{split}
  (-1+ (g(\frac{\p_{y}}{ik}-t))^{2})\psi &= R,  \\
  \psi|_{y=0,1}&=0, \\
  y &\in [0,1] \\
  C> g^{2}>c>0, g &\in W^{2,\infty}.
  \end{split}
\end{align*}
for some $R \in H^{s}, 0\leq s <1/2$.

In the following we show that, as in the case $s=0$, for $|k^{-1}|$  sufficiently small
\begin{align*}
  \|\psi\|_{H^{s}}^{2} + \|(\frac{\p_{y}}{ik}-t)\psi\|_{H^{s}}^{2} \lesssim \sum c_{n}(t) <n>^{2s}|R_{n}|^{2},
\end{align*}
for some family $c_{n}(t) \in L^{1}_{t}$ with $\|c_{n}(t)\|_{L^{1}_{t}} <C<\infty$  uniformly in $n$.

As in the case $s=0$, the heuristic idea is to consider the inner product (now in $H^{s}$) of the first equation in \eqref{eq:ELL} with $\psi$ and estimate:
\begin{align*}
\tag{lower}
\label{eq:lower}
  \|\psi\|_{H^{s}}^{2} + \|(\frac{\p_{y}}{ik}-t)\psi\|_{H^{s}}^{2} \lesssim & \Re \langle \psi, R \rangle_{H^{s}} - \text{errors} ,\\
\tag{upper}
\Re \langle \psi, R \rangle_{H^{s}} 
  \lesssim & \left(\|\psi\|_{H^{s}}^{2} + \|(\frac{\p_{y}}{ik}-t)\psi\|_{H^{s}}^{2}\right)^{1/2}
\left(\sum c_{n}(t) <n>^{2s}|R_{n}|^{2}\right)^{1/2}.
\end{align*}
Here, errors are terms that can either be absorbed in the left-hand-side or estimated by terms similar to the right-hand-side in (upper).

As we work in fractional Sobolev spaces, integration by parts and similar estimates involve many more boundary terms, commutators and other corrections.
Controlling all these terms in a suitable way, makes \eqref{eq:lower} technically much more challenging than in the integer Sobolev case.
The upper estimate, however, follows analogously, as is shown in the following lemma.

\begin{lem}
\label{lem:another}
  Let $\psi, R$ solve \eqref{eq:ELL}, then 
  \begin{align*}
 \Re \langle \psi, R \rangle_{H^{s}} 
 \lesssim  \left(\|\psi\|_{H^{s}}^{2} + \|(\frac{\p_{y}}{ik}-t)\psi\|_{H^{s}}^{2}\right)^{1/2}
\left(\sum c_{n}(t) <n>^{2s}|R_{n}|^{2}\right)^{1/2}, 
  \end{align*}
where 
\begin{align*}
  c_{n}(t) = \frac{1}{1+(\frac{n}{k}-t)^{2}} \in L^{1}_{t}.
\end{align*}
\end{lem}

\begin{proof}[Proof of Lemma \ref{lem:another}]
Following the same strategy as in Lemma \ref{lem:easy}, we express $\langle \psi, R \rangle_{H^{s}}$ in a basis, multiply by a factor
\begin{align*}
  \frac{1+i(n/k-t)}{1+i(n/k-t)},
\end{align*}
 integrate by parts and employ Cauchy-Schwarz.
\end{proof}

In order to derive \eqref{eq:lower}, we first make use of our freedom in choosing the error term, by modifying the (shifted) elliptic operator.
\begin{align*}
 & (-1+ (g(\frac{\p_{y}}{ik}-t))^{2})\psi \\
=& - \psi + (\frac{\p_{y}}{ik}-t)g^{2}(\frac{\p_{y}}{ik}-t)\psi - \frac{g'}{ik}g(\frac{\p_{y}}{ik}-t)\psi.
\end{align*}
Up to boundary terms, the leading operator 
\begin{align*}
  -1 + (\frac{\p_{y}}{ik}-t)g^{2}(\frac{\p_{y}}{ik}-t)
\end{align*}
is hence symmetric and negative definite, which we use for a lower estimate in Lemma \ref{lem:reductionboundary} and in combination with Proposition \ref{prop:ComHs}.

\begin{lem}
\label{lem:commute}
  Let $\psi \in H^{s}([0,1])$ be a solution of \eqref{eq:ELL}. Then, 
  \begin{align*}
   \left| \left\langle \psi,\frac{g'}{ik}g\left(\frac{\p_{y}}{ik}-t\right)\psi\right\rangle_{H^{s}} \right |
\lesssim \frac{1}{|k|}| \|\psi\|_{H^{s}} \|(\frac{\p_{y}}{ik}-t)\psi \|_{H^{s}}^{2}.
  \end{align*}
  
  For $k$ sufficiently large, instead of \eqref{eq:lower}, it thus suffices to prove
  \begin{align*}
   \|\psi\|_{H^{s}}^{2} + \|(\frac{\p_{y}}{k}-it)\psi\|_{H^{s}}^{2}   \lesssim \langle \psi, - \psi + (\frac{\p_{y}}{ik}-t)g^{2}(\frac{\p_{y}}{ik}-t)\psi \rangle_{H^{s}} - \text{errors}.
  \end{align*}
\end{lem}

\begin{proof}[Proof of Lemma \ref{lem:commute}]
  The first statement follows by Cauchy-Schwarz and applying Proposition \ref{prop:LipHs} of Section \ref{sec:fract-sobol-space} with $gg' \in W^{1,\infty}(\T)$.

For the second statement, we note that 
\begin{align*}
  c(\|\psi\|_{H^{s}}^{2} + \|(\frac{\p_{y}}{k}-it)\psi\|_{H^{s}}^{2}) &\leq \Re \langle \psi,R \rangle - \text{errors} \\
&= \Re \langle \psi, - \psi + (\frac{\p_{y}}{ik}-t)g^{2}(\frac{\p_{y}}{ik}-t)\psi \rangle_{H^{s}} - \text{errors} \\
&\quad + \Re \langle \psi,\frac{g'}{ik}g(\frac{\p_{y}}{ik}-t)\psi\rangle_{H^{s}}  \\
&\leq  \Re \langle \psi, - \psi + (\frac{\p_{y}}{ik}-t)g^{2}(\frac{\p_{y}}{ik}-t)\psi \rangle_{H^{s}} - \text{errors}  \\ &\quad + \frac{C}{|k|} ( \|\psi\|_{H^{s}}^{2} + \|(\frac{\p_{y}}{k}-it)\psi\|_{H^{s}}^{2}).
\end{align*}
Letting $|k|\gg 0$ be sufficiently large, $\frac{C}{k} \leq c/2$, which allows us to absorb the last term in the left-hand-side.
\end{proof}

In order to prove \eqref{eq:lower}, it thus remains to show that 
\begin{align*}
 -\Re \langle \psi, (\frac{\p_{y}}{ik}-t)g^{2}(\frac{\p_{y}}{ik}-t)\psi \rangle_{H^{s}}
\end{align*}
provides a control of 
\begin{align*}
  \|(\frac{\p_{y}}{ik}-t)\psi\|_{H^{s}}^{2},
\end{align*}
up to error terms.

While in the case $s=0$ this reduces to an integration by parts argument, for $s>0$ two additional challenges arise:
\begin{itemize}
\item Integrating by parts yields boundary terms. 
\item $\langle u,g^{2}u \rangle_{H^{s}} \neq \langle gu,gu \rangle_{H^{s}} \not \geq \min(g^{2})\|u\|_{H^{s}}^{2}$.
\end{itemize}
The second issue is addressed by Proposition \ref{prop:ComHs} in Section \ref{sec:fract-sobol-space} and the former by the following two lemmata.

\begin{lem}
\label{lem:reductionboundary}
  Let $\psi \in H^{s}([0,1])$ be a solution of \eqref{eq:ELL}. Then 
  \begin{align*}
    \left|\left\langle \psi,\left(\frac{\p_{y}}{ik}-t\right)g^{2}\left(\frac{\p_{y}}{ik}-t\right)\psi  \right\rangle_{H^{s}} +  \left\langle \left(\frac{\p_{y}}{ik}-t\right)\psi, g^{2}\left(\frac{\p_{y}}{ik}-t\right)\psi  \right\rangle_{H^{s}} \right| \\
\lesssim |k^{-1}| \left(\|\psi\|_{H^{s}}^{2} + \left\|\left(\frac{\p_{y}}{ik}-t\right)\psi\right\|_{H^{s}}^{2}\right)^{1/2} \left\|\frac{<n>^{s}}{<n/k-t>}\right\|_{l^{2}} \left|g^{2} \left(\frac{\p_y}{k}-it\right)\psi|_{y=0}^{1} \right|.
  \end{align*}
Furthermore, 
\begin{align*}
  \left\|\frac{<n>^{s}}{<n/k-t>}\right\|_{l^{2}} &\lesssim_{s} <kt>^{s}. \\
\end{align*}

\end{lem}

\begin{proof}[Proof of Lemma \ref{lem:reductionboundary}]
Expanding both terms in a Fourier basis and integrating by parts, the difference is given by
  \begin{align*}
  \left.\sum_n <n>^{2s} \psi_{n} \frac{1}{k} g^{2} (\frac{\p_y}{k}-it)\psi\right|_{y=0}^{1}.
  \end{align*}
Taking absolute values inside the sum, multiplying by a factor 
\begin{align*}
  1= \frac{1+i(n/k-t)}{1+i(n/k-t)}
\end{align*}
and using Cauchy-Schwarz, the first estimate is proven.

For the second estimate, we note that 
\begin{align*}
  <n>^{s} \lesssim k^{s} <n/k -t>^{s} + <kt>^{s},
\end{align*}
and that
\begin{align*}
  <n/k -t>^{s-1} \in l^{2}_{n},
\end{align*}
provided $s<1/2$.

\end{proof}

\begin{lem}
\label{lem:boundaryestimates}
  Let $\psi, R$ solve \eqref{eq:ELL}, then the following estimates hold:
  \begin{align*}
\tag{a}
  \left|g^{2} \left(\frac{\p_y}{k}-it\right)\psi|_{y=0}^{1} \right| &\lesssim |k|^{-1}<t>^{-s} \left( \sum_{n} |R_{n}|^{2}c_{n}(t)<n>^{2s} \right)^{1/2}.
  \end{align*}

\begin{align*}
\tag{b}
  g^{2}\left(\frac{\p_{y}}{ik}-t\right) \psi|_{y=0} &= k\langle R, e^{ikty} u_{1} \rangle_{L^{2}}, \\
g^{2}\left(\frac{\p_{y}}{ik}-t\right) \psi|_{y=1} &= k\langle R, e^{ikt(y-1)} u_{2} \rangle_{L^{2}},
\end{align*}
\begin{align*}
\tag{c}
  |\langle R, e^{ikty} u_{1} \rangle_{L^{2}}| &\lesssim \sum_{n} |R_{n}|<\frac{n}{k}-t>^{-1}, \\
  |\langle R, e^{ikt(y-1)} u_{2} \rangle_{L^{2}}| &\lesssim \sum_{n} |R_{n}|<\frac{n}{k}-t>^{-1}, \\
|\langle R, e^{ikt(y-1)} u_{2} \rangle_{L^{2}} + \langle R, e^{ikty} u_{1} \rangle_{L^{2}} | 
&\lesssim |k^{-1}|\sum_{n} |R_{n}|<\frac{n}{k}-t>^{-2}.
\end{align*}

\end{lem}

\begin{proof}[Proof of Lemma \ref{lem:boundaryestimates}]
  
  We first show that that (b) and (c) imply (a).
  Thus, assume for the moment, that (c) holds. Then
  \begin{align*}
    \left|g^{2} (\frac{\p_y}{k}-it)\psi|_{y=0}^{1} \right| \lesssim & |k^{-1}|\sum_{n} |R_{n}|<\frac{n}{k}-t>^{-2} \\
= & |k^{-1}| \sum_{n} |R_{n}|\frac{<n>^{s}}{<n/k-t>^{1/2+\epsilon}} \frac{1}{<n/k-t>^{1/2+\epsilon}} \frac{1}{<n>^{s}<n/k-t>^{1-2\epsilon}} \\
\leq & |k^{-1}| \left( \sum_{n} |R_{n}|^{2}c_{n}(t)<n>^{2s} \right)^{1/2} \\
&\left\| \frac{1}{<n/k-t>^{1/2+\epsilon}}\right\|_{l^{2}} \left\|\frac{1}{<n>^{s}<n/k-t>^{1-2\epsilon}}\right\|_{l^{\infty}},
  \end{align*}
where
\begin{align*}
  c_{n}(t) = <n/k-t>^{-1-2\epsilon} \in L^{1}_{t}.
\end{align*}
We further estimate 
\begin{align*}
  \| \frac{1}{<n/k-t>^{1/2+\epsilon}}\|_{l^{2}} &\lesssim \sqrt{k}, \\
\|\frac{1}{<n>^{s}<n/k-t>^{1-2\epsilon}}\|_{l^{\infty}} &\leq <kt>^{-s} + <kt>^{-1+2\epsilon}.
\end{align*}
As $s<1/2<1$, for $\epsilon>0$ sufficiently small $1-2\epsilon \geq s$, which concludes the proof of (a).
\\

The estimates (b) have been proven previously in Lemma \ref{lem:boundary32} for the case of $\psi=\Phi$.
Let again $e^{ikty}u_{1}$, $e^{ikt(y-1)}u_{2}$ be the homogeneous solutions with boundary values zero and one.
Testing the equation and integrating by parts twice, yields two boundary terms.
In the case of $e^{ikty}u_{1}$, the first boundary term is given by 
\begin{align*}
 e^{ikty}u_{1}\frac{1}{ik} g^{2}(\frac{\p_{y}}{ik}-t) \psi |_{y=0}^{1} = -\frac{1}{ik} g^{2} (\frac{\p_{y}}{ik}-t) \psi |_{y=0},
\end{align*}
by the choice of the boundary values of $e^{ikty}u_{1}$.
The second boundary term 
\begin{align*}
 \psi \frac{1}{ik} g^{2}(\frac{\p_{y}}{ik}-t)e^{ikty}u_{1}  |_{y=0}^{1},
\end{align*}
vanishes as $\psi$ vanishes on the boundary.
The result for $e^{ikt(y-1)}u_{2}$ follows analogously, which concludes the proof of (b).
\\

It remains to prove (c).
For the first two estimates, it suffices to prove that 
\begin{align*}
  \langle e^{iny}, e^{ikty}u_{1} \rangle_{L^{2}} &\lesssim <n/k -t>^{-1}, \\
  \langle e^{iny}, e^{ikt(y-1)}u_{2} \rangle_{L^{2}} &\lesssim <n/k -t>^{-1}.
\end{align*}
A first, easy but non-optimal proof integrates $e^{i(kt-n)y}$ by parts, which yields a control by 
\begin{align*}
  \left|\frac{k}{kt-n} \right|.
\end{align*}
For an improved estimate we recall that $u_{j}$ is given by linear combinations of 
\begin{align*}
  e^{\pm k U^{-1}(y)},
\end{align*}
and that 
\begin{align*}
  e^{i(kt-n)y \pm k U^{-1}(y)} = \frac{1}{\pm k(U^{-1})'+ i(kt-n)} \p_{y}e^{i(kt-n)y \pm k U^{-1}(y)} .
\end{align*}

The improved final estimate of (c), follows by noting that $e^{ikt(y-1)}u_{2}+ e^{ikty}u_{1}$ has boundary values $1,1$ and is thus periodic.
A first integration by parts thus does not yield any boundary contribution and we may integrate by parts once more to obtain the quadratic decay. 

\end{proof}

Combining both lemmata, we thus have further simplified \eqref{eq:lower} to estimating 
\begin{align*}
  \left\langle \left(\frac{\p_{y}}{ik}-t\right)\psi, g^{2} \left(\frac{\p_{y}}{ik}-t\right)\psi \right\rangle_{H^{s}} .
\end{align*}
Employing Proposition \ref{prop:ComHs} of Section \ref{sec:fract-sobol-space}, as well as the $L^{2}$ stability result of \cite{Zill3}, Theorem \ref{thm:citedL2result}, we have thus proven the following proposition: 

\begin{prop}
\label{prop:A}
 Let $\psi, R$ solve \eqref{eq:ELL}, $0\leq s <1/2$ and $R \in H^{s}$.
Then 
 \begin{align*}
   \|\psi\|_{H^{s}}^{2} + \|(\frac{\p_{y}}{k}-it)\psi\|_{H^{s}} \lesssim \sum_{n} |R_{n}|^{2}c_{n}(t)<n>^{2s},
 \end{align*}
 where $c_{n}\in L^{1}_{t}$ with $\|c_{n}\|_{L^{1}_{t}}$ bounded uniformly in $n$.
\end{prop}

Having derived this generic result for \eqref{eq:ELL}, it remains to apply it to the cases $\psi=\Phi$ and $\psi= \Phi^{(1)}$.
\begin{prop}
\label{prop:B}
  Let $0<s<1/2$, $W \in H^{s}$ and let $\Phi$ be a solution of 
  \begin{align*}
    (-k^{2}+ (g(\p_{y}-ikt))^{2}) \Phi &= W , \\
    \Phi|_{y=0,1}&= 0, \\
    y &\in [0,1] .
  \end{align*}
  Let further $g,g' \in W^{1,\infty}(\T)$ and $g^{2}>c>0$.
Then there exists a constant such that
\begin{align*}
\|\Phi\|_{H^{s}}^{2} + \|(\frac{\p_{y}}{k}-it)^{2}\Phi\|_{H^{s}}^{2} \lesssim  \sum_{n} |W_{n}|^{2} <n>^{2}c_{n}(t),
\end{align*}
for some $c_{n}(t) \in L^{1}_{t}$.
\end{prop}

\begin{proof}[Proof of Proposition \ref{prop:B}]
  Applying Proposition \ref{prop:A} with $\psi=\Phi$, $R=W$ yields the result.
\end{proof}

Considering the case $\psi= \Phi^{(1)}$, the upper estimate, Lemma \ref{lem:easy}, has to be slightly modified, as the second term in 
\begin{align*}
  R= \p_{y}W + \left[(\frac{\p_{y}}{k}-it)g^{2}(\frac{\p_{y}}{k}-it), \p_{y}\right]\Phi
\end{align*}
has to be treated separately.

\begin{lem}
\label{lem:lastH32lemma}
  Let $\Phi, W$ solve 
  \begin{align*}
    (-1+ (g(\frac{\p_{y}}{k}-it))^{2})\Phi &= W, \\
    \Phi|_{y=0,1}&=  0, \\
    y &\in [0,1].
  \end{align*}
  Then, 
  \begin{align*}
    \Re \langle \Phi^{(1)}, [(\frac{\p_{y}}{k}-it)g^{2}(\frac{\p_{y}}{k}-it), \p_{y}]\Phi  \rangle_{H^{s}} \lesssim 
\sum_{n}<n>^{2s}c_{n}(t) (|(\p_{y}W)_{n}|^{2}+ |W_{n}|^{2}).
  \end{align*}
\end{lem}

\begin{proof}[Proof of Lemma \ref{lem:lastH32lemma}]
  We compute 
  \begin{align*}
    \left[(\frac{\p_{y}}{k}-it)g^{2}(\frac{\p_{y}}{k}-it), \p_{y}\right]\Phi = 2 (\frac{\p_{y}}{k}-it)gg'(\frac{\p_{y}}{k}-it) \Phi.
  \end{align*}
  Integrating by parts, we thus obtain a bulk term 
  \begin{align*}
    \langle (\frac{\p_{y}}{k}-it) \Phi^{(1)}, 2gg' (\frac{\p_{y}}{k}-it) \Phi \rangle_{H^{s}},
  \end{align*}
  and, similar to Lemma \ref{lem:commute}, a boundary term 
  \begin{align}
\label{eq:900}
    \sum_{n} \Phi^{(1)}_{n}<n>^{2s} k^{-1} 2gg' (\frac{\p_{y}}{k}-it) \Phi .
  \end{align}
Using Proposition \ref{prop:LipHs} of Section \ref{sec:fract-sobol-space} and Young's inequality, the bulk term can be estimated by
  \begin{align*}
    \epsilon \|(\frac{\p_{y}}{k}-it) \Phi^{(1)}\|_{H^{s}}^{2} + \epsilon^{-1} C \|(\frac{\p_{y}}{k}-it) \Phi\|_{H^{s}}^{2}.
  \end{align*}
  Here, the second term can be estimated by Proposition \ref{prop:B}, while the first term can be absorbed in the left-hand-side of the estimate as in the proof of Lemma \ref{lem:commute}.

In order to estimate the boundary term, \eqref{eq:900}, we follow the same strategy as in the proof of Lemma \ref{lem:reductionboundary} and Lemma \ref{lem:boundaryestimates}.
We thus obtain an estimate by 
\begin{align*}
  \|(\frac{\p_{y}}{k}-it) \Phi^{(1)}\|_{H^{s}} \|\frac{<n>^{s}}{<n/k-t>}\|_{l^{2}} 
\left| 2gg'(\frac{\p_{y}}{k}-it) \Phi |_{y=0}^{1}  \right| . 
\end{align*}
It remains to estimate 
\begin{align*}
  \left| 2gg'(\frac{\p_{y}}{k}-it) \Phi |_{y=0}^{1}  \right| .
\end{align*}
Unlike in the last case of (c) in Lemma \ref{lem:boundaryestimates}, there is no additional cancellation of the contributions at $y=0$ and $y=1$. Hence, we estimate 
\begin{align*}
  |2gg'| \lesssim \|g\|_{W^{1,\infty}}^{2}
\end{align*}
and consider the contributions at $y=0$ and $y=1$ separately.
Using Lemma \ref{lem:boundaryestimates}, we express 
\begin{align*}
  (\frac{\p_{y}}{k}-it) \Phi |_{y=0,1}
\end{align*}
in terms of 
\begin{align*}
  \langle W, e^{ikty}u_{1} \rangle_{L^{2}} = \frac{1}{ikt}W|_{y=0} + \frac{1}{ikt}\langle e^{ikty} \p_{y}Wu_{1} \rangle_{L^{2}} .
\end{align*}
To estimate both terms, we follow the same strategy as in the proof of Theorem \ref{thm:H32boundary}.
The first term is controlled using Young's inequality, i.e.
\begin{align*}
  \|(\frac{\p_{y}}{k}-it) \Phi^{(1)}\|_{H^{s}} \|\frac{<n>^{s}}{<n/k-t>}\|_{l^{2}} \frac{W|_{y=0,1}}{ikt} \\
\lesssim \epsilon^{-1} |kt|^{-2\gamma} +  \epsilon \|(\frac{\p_{y}}{k}-it) \Phi^{(1)}\|_{H^{s}}^{2}\frac{<kt>^{2s}}{|kt|^{2(1-\gamma)}},
\end{align*}
where $\gamma>1/2$ is chosen such that $1-\gamma \geq s$.
The first term is integrable in time and the second can be absorbed in the left-hand-side.

It remains to estimate
\begin{align}
\label{eq:localmark}
  \|(\frac{\p_{y}}{k}-it) \Phi^{(1)}\|_{H^{s}} \|\frac{<n>^{s}}{<n/k-t>}\|_{l^{2}} \frac{1}{ikt}\langle e^{ikty} \p_{y}Wu_{1} \rangle_{L^{2}} .
\end{align}
For this purpose, we compute 
\begin{align*}
\langle e^{ikty}, \p_{y}Wu_{1} \rangle_{L^{2}} =  \langle e^{ikty},u_{1} \p_{y}W \rangle_{L^{2}} + \langle e^{ikty},W \p_{y}u_{1} \rangle_{L^{2}}.
\end{align*}
The second term can be integrated by parts once more to obtain another factor $\frac{1}{ikt}$ and is thus easily controlled.
For the first term we estimate 
\begin{align*}
  \langle e^{ikty},u_{1} \p_{y}W \rangle_{L^{2}} \lesssim \sum |(\p_{y}W)_{n}| \frac{<n>^{s}}{<n/k-t>^{1-\lambda}} \frac{1}{<n>^{s}<n/k-t>^{\lambda}},
\end{align*}
where $0<\lambda<1$ and $s+\lambda>1/2$.

The terms in \eqref{eq:localmark} can thus be estimated by 
\begin{align*}
& \|(\frac{\p_{y}}{k}-it) \Phi^{(1)}\|_{H^{s}} <kt>^{s} \frac{1}{|kt|} \left\| |(\p_{y}W)_{n}| \frac{<n>^{s}}{<n/k-t>^{1-\lambda}}\right\|_{l^{2}} \left\| \frac{1}{<n>^{s}<n/k-t>^{\lambda}}\right\|_{l^{2}} \\
\lesssim &\|(\frac{\p_{y}}{k}-it) \Phi^{(1)}\|_{H^{s}} \frac{1}{|kt|} \left\| |(\p_{y}W)_{n}| \frac{<n>^{s}}{<n/k-t>^{1-\lambda}}\right\|_{l^{2}}.
\end{align*}

Using Young's inequality, the first factor can be absorbed, while the second factor is of the desired form with 
\begin{align*}
  c_{n}(t) :=\frac{1}{|kt|} \frac{1}{<n/k-t>^{2(1-\lambda)}} \in L^{1}_{t}.
\end{align*}
\end{proof}
This concludes the stability proof in $H^{s}, s<3/2$.

As a consequence we now have sufficient control of regularity to obtain damping with integrable rates and scattering.

\begin{cor}[Scattering]
\label{cor:scat1}
  Let $0<s<1/2$ and let $W$ be a solution of the linearized Euler equations, \eqref{eq:Eulernochmal}, such that $\|\p_{y}W\|_{H^{s}}$ and  $\|W\|_{H^{1}}$ are uniformly bounded (e.g. satisfying Theorem \ref{thm:H32}).
 Then there exists $W^{\infty} \in H^{s}_{y}L^{2}_{x}$ such that, as $t \rightarrow \infty$,
  \begin{align*}
    \|V_{2}\|_{L^{2}} &= \mathcal{O}(t^{-(1+s)}), \\
 W &\xrightarrow{L^{2}} W_{\infty} , \\
     \|W(t) - W_{\infty}\|_{L^{2}} &= \mathcal{O}(t^{-s}). 
  \end{align*}
\end{cor}

\begin{proof}[Proof of Corollary \ref{cor:scat1}]
  Applying Duhamel's formula, i.e. integrating the equation in time, $W(t)$ satisfies 
  \begin{align*}
    W(t)= \omega_{0} + \int^{t}_{0} f V_{2}(\tau) d\tau.
  \end{align*}
  Estimating and integrating,
  \begin{align*}
    \|f V_{2}(\tau)\|_{L^{2}} \leq \|f\|_{L^{\infty}} \|V_{2}\|_{L^{2}}=\mathcal{O}(t^{-(1+s)}),
  \end{align*}
  then yields the result.
\end{proof}

Approximating $\omega_{0}\in L^{2}$ by functions in $H^{s}, 1<s<3/2$, we obtain scattering in $L^{2}$.
\begin{cor}[$L^{2}$ scattering]
\label{cor:scat2}
  Let $f,g,k$ be as in Theorem \ref{thm:H32}. Then for any $\omega_{0} \in L^{2} $ there exists $W_{\infty} \in L^{2}$ such that 
  \begin{align*}
    W \xrightarrow{L^{2}} W_{\infty},
  \end{align*}
as $t \rightarrow \infty$.
\end{cor}

\begin{proof}[Proof of Corollary \ref{cor:scat2}]
  Let $ (\omega^{n}_{0})_{n \in \N} \in H^{s} $  be a sequence such that 
  \begin{align*}
\omega^{n}_{0} \xrightarrow{L^{2}} \omega_{0},
  \end{align*}
as $n \rightarrow \infty$.
  By Corollary \ref{cor:scat2}, for any $\omega_{0}^{n}$ there exists an asymptotic profile $W^{n}_{\infty} $.
  By the $L^{2} $ stability theorem of \cite{Zill3}, Theorem \ref{thm:citedL2result}, the convergence of $\omega^{n}_{0}$ also implies the convergence of $W^{n}(t)$ at any time $t$ and of $W^{n}_{\infty} $.
The result then follows by choosing an appropriate diagonal sequence in $t$ and $n$.
\end{proof}

\subsection{Stability in $H^{5/2-}$}
\label{sec:stability-h52-}

In the previous Section \ref{sec:stability-h32-}, we have seen that, under general perturbations, the critical Sobolev exponent in $y$ is given by $s=\frac{3}{2}$.
More precisely, for any $m \in \N_{0}$, we have shown stability  in the periodic fractional Sobolev spaces 
\begin{align*}
 H^{m}_{x}H^{s}_{y}(\T_{L}\times \T),s<\frac{3}{2}, 
\end{align*}
and that stability in 
\begin{align*}
H^{m}_{x}H^{s}_{y}(\T_{L}\times[0,1]),s>\frac{3}{2}, 
\end{align*}
can in general not hold, unless one restricts to initial perturbations $\omega_{0}$ with zero Dirichlet boundary data, $\omega_{0}|_{y=0,1}=0$.

Restricting to such perturbations, in \cite{Zill3} we established stability in $H^{m}_{x}H^{2}_{y}(\T_{L} \times [0,1])$, which is sufficient to prove linear inviscid damping with the optimal algebraic rates.
However, $H^{2}$ stability is not sufficient to establish consistency with the nonlinear equations, since control of the nonlinearity, 
\begin{align*}
  \nabla^{\bot}\Phi \cdot \nabla W,
\end{align*}
would require an $L^{\infty}$ control of $\nabla W$. As we work in two dimensions, in order to use a Sobolev embedding, we thus require control in $H^{s},s>2$.

As the main result of this section, we hence show that, for this restricted class of perturbations, $\omega_{0}$, the critical Sobolev exponent in $y$ is given by $s=\frac{5}{2}$.
More precisely, as shown in Corollary \ref{cor:H2boundaryquadratic}, for initial perturbations, $\omega_{0}$, with zero Dirichlet data, $\omega_{0}|_{y=0,1}=0$, generically $\p_{y}^{2}W$ asymptotically develops (logarithmic) singularities a the boundary. Hence, even for this restricted class of perturbations, stability in $H^{m}_{x}H^{s}_{y}(\T_{L}\times [0,1]), s>\frac{5}{2}$, can in general not hold.
As we discuss in Section \ref{cha:anoth-sect-cons}, this further implies instability of the nonlinear problem in the finite periodic channel in high Sobolev spaces and therefore, in particular, forbids nonlinear inviscid damping results in Gevrey regularity such as in the work of Bedrossian and Masmoudi, \cite{bedrossian2013inviscid}.

As a complementary result to the instability, Theorem \ref{thm:H52} establishes stability in the periodic fractional Sobolev spaces, $H^{m}_{x}H^{s}_{y}(\T_{L}\times\T),s<\frac{5}{2}$. This additional stability allows us to prove consistency with the nonlinear problem, also for the finite periodic channel. 
\\
 
We recall that the linearized Euler equations, \eqref{eq:Eulernochmal}, decouple with respect to $k$ and we may hence consider $k$ as a given parameter and consider the stability of 
\begin{align*}
  W(t)=W(t,k,\cdot) \in H^{s}([0,1]) \text{ or } H^{s}(\T).
\end{align*}

The following two lemmata provide a description of the evolution of derivatives of $\Phi$ on the boundary. Using these lemmata, in Proposition \ref{cor:H2boundaryquadratic} we show that, in general, stability in $H^{s}([0,1]), s>\frac{5}{2}$, can not hold. 
\begin{lem}
\label{lem:pWcontrol}
  Let $W$ be a solution of the linearized Euler equations, \eqref{eq:Eulernochmal}, and suppose that $\|W\|_{H^{2}([0,1])}$ is bounded uniformly in time.
  Suppose further that $\omega_{0}|_{y=0,1}\equiv 0$.
  Then there exist constants $c_{0}, c_{1} \in \C$ such that  
  \begin{align*}
    \p_{y}W|_{y=0} \rightarrow c_{0}, \\ 
    \p_{y}W|_{y=1} \rightarrow c_{1}, 
  \end{align*}
  as $t \rightarrow \infty$.
\end{lem}
We remark that $c_{0}, c_{1}$ are in general non-trivial and not determined by $\p_{y}\omega_{0}|_{y=0,1}$.
In analogy to Corollary \ref{cor:not32}, in Corollary \ref{cor:H2boundaryquadratic} we show that non-trivial $c_{0},c_{1}$ asymptotically result in a (logarithmic) blow-up at the boundary and thus provide an upper limit on stability results.

\begin{proof}[Proof of Lemma \ref{lem:pWcontrol}]
  Restricting the evolution equation for $\p_{y}W$, \eqref{eq:pyW}, to the boundary, we obtain 
  \begin{align*}
    \dt \p_{y}W|_{y=0,1} = \frac{if}{k} \p_{y}\Phi|_{y=0,1},
  \end{align*}
  where we used that $\Phi|_{y=0,1}\equiv 0$.
  It therefore suffices to show that $\p_{y}\Phi|_{y=0,1}$ decays in $t$ at an integrable rate.
We recall that by Lemma \ref{lem:boundaryestimates}
  \begin{align*}
    \p_{y}\Phi|_{y=0} &= \frac{k}{g(0)} \langle W, e^{ikty}u_{1}(0,y) \rangle_{L^{2}}, \\
    \p_{y}\Phi|_{y=1} &= \frac{k}{g(1)} \langle W, e^{ikt(y-1)}u_{2}(0,y) \rangle_{L^{2}}.
  \end{align*}
  As $k \neq 0$ and as $g$ is bounded away from $0$, it suffices to consider the $L^{2}$ products.  
  Integrating by parts once, we obtain 
  \begin{align*}
    \langle W, e^{ikty}u_{1}(0,y) \rangle_{L^{2}} &= -\frac{1}{ikt} W|_{y=0} - \frac{1}{ikt} \langle e^{ikty}, \p_{y} (W u_{1}(0,y))  \rangle_{L^{2}} \\
&= - \frac{1}{ikt} \langle e^{ikty}, \p_{y} (W u_{1}(0,y))  \rangle_{L^{2}}.
  \end{align*}
  Recalling Lemma \ref{lem:boundary32}, a uniform control of $\|W\|_{H^{s}}+ \|\p_{y}W\|_{H^{s}}$ for some $s>0$ suffices to obtain an upper bound by $\mathcal{O}(t^{-1-s})$
  and thus deduce the result.
  
  Integrating by parts once more, we obtain 
  \begin{align*}
\langle W, e^{ikty}u_{1}(0,y) \rangle_{L^{2}}= \frac{1}{k^{2}t^{2}} e^{ikty} \p_{y} (W u_{1}(0,y))|_{y=0}^{1} - \frac{1}{k^{2}t^{2}} \langle e^{ikty}, \p_{y}^{2} (W u_{1}(0,y))  \rangle_{L^{2}}. 
  \end{align*}
  Again using the assumption that $W|_{y=0,1}\equiv 0$, the first term can be controlled by 
  \begin{align*}
    C_{k} t^{-2} |\p_{y}W|_{y=0,1}|,
  \end{align*}
  and the second term by 
  \begin{align*}
    C_{k}t^{-2} \|W\|_{H^{2}}^{2}. 
  \end{align*}
  Using the uniform control of $\|W\|_{H^{2}}$, we thus obtain the differential inequality
  \begin{align*}
    |\dt \p_{y}W|_{y=0,1}| \lesssim t^{-2} (|\p_{y}W|_{y=0,1}|+1) .
  \end{align*}
  Integrating this inequality then yields the result.
\end{proof}

Following a similar approach as in Section \ref{sec:stability-h32-}, we show that $\p_{y}^{2}W|_{y=0,1}$ in general grows unboundedly as $t \rightarrow \infty$.

\begin{lem}
  Let $W$ be a solution of the linearized Euler equations, \eqref{eq:Eulernochmal}, and suppose that, for some $s>0$, $\|W(t)\|_{H^{2}}$ and $\|\p_{y}^{2}W(t)\|_{H^{s}}$ are bounded uniformly in time.
  Then, as $t\rightarrow \infty$,
  \begin{align*}
    \p_{y}^{2}\Phi|_{y=0,1}= \frac{1}{ikt}\p_{y}W|_{y=0,1} + \mathcal{O}(t^{-1-s}).
  \end{align*}
\end{lem}

\begin{proof}
  Following the same approach as in the proof of Lemma \ref{lem:boundary32}, we note that by \eqref{eq:Eulernochmal},
  \begin{align*} 
   (-1+(g(y)(\frac{\p_{y}}{k} -it))^{2}) \Phi &= W,
  \end{align*}
  and by the choice of zero Dirichlet boundary values of $\Phi$ and $W$,
  \begin{align*}
    g^{2}\p_{y}^{2}\Phi|_{y=0,1}=(-gg'+iktg^{2})\p_{y}\Phi|_{y=0,1}.
  \end{align*}
  Dividing by $g^{2}$ and using
    \begin{align*}
  \p_{y}\Phi|_{y=0}&= \frac{k}{g^{2}(0)} \langle W, e^{ikty}u_{1} \rangle \\
&=  \frac{k}{g^{2}(0)} \left( \frac{1}{ikt} \omega_{0}|_{y=0} +  \langle e^{ikty}, \p_{y} W u_{1} \rangle \right) , \\
  \p_{y}\Phi|_{y=1}&= \frac{k}{g^{2}(1)} \langle W, e^{ikt(y-1)}u_{2} \rangle \\
&=\frac{k}{g^{2}(1)} \left( \frac{1}{ikt} \omega_{0}|_{y=1} +  \langle e^{ikt(y-1)}, \p_{y} W u_{2} \rangle \right),
\end{align*}
from Lemma \ref{lem:boundary32}, it thus suffices to consider 
\begin{align*}
  \langle e^{ikty}, \p_{y} W u_{1} \rangle, \\
  \langle e^{ikt(y-1)}, \p_{y} W u_{2} \rangle .
\end{align*}
Integrating $e^{ikty}$ or $e^{ikt(y-1)}$ by parts and using boundary values of $u_{1},u_{2}$, yields the leading terms
\begin{align*}
  \frac{1}{ikt} \p_{y}W|_{y=0,1},
\end{align*}
while the remainder is given by 
\begin{align*}
  \frac{1}{ikt}\langle e^{ikty}, \p_{y}(\p_{y} W u_{1}) \rangle, \\ 
  \frac{1}{ikt}\langle e^{ikt(y-1)}, \p_{y}( \p_{y} W u_{2})\rangle,
\end{align*}
respectively.
By the product rule
\begin{align*}
  \p_{y}(\p_{y} W u_{j})= u_{j}\p_{y}^{2}W + \p_{y}W \p_{y}u_{j}.
\end{align*}
For the latter term integrating by parts once more yields a term controlled by 
\begin{align*}
  \mathcal{O}((kt)^{-2}) \|W\|_{H^{2}}.
\end{align*}
It thus suffices to consider only 
\begin{align*}
  \frac{1}{ikt}\langle e^{ikty}u_{1}, \p_{y}^{2} W  \rangle, \\
  \frac{1}{ikt}\langle e^{ikt(y-1)}u_{2}, \p_{y}^{2} W  \rangle.
\end{align*}
Expanding into a basis and using duality, the result then follows by estimating 
\begin{align*}
  \| e^{ikty}u_{1}\|_{H^{-s}} + \|e^{ikt(y-1)}u_{2}\|_{H^{-s}} =\mathcal{O}(t^{-s}).
\end{align*}
\end{proof}

\begin{cor}
\label{cor:H2boundaryquadratic}
  Let $\omega_{0}|_{y=0,1}\equiv 0$ and let $W$ be the solution of \eqref{eq:Eulernochmal}. Further suppose that the limits
  \begin{align*}
    \lim_{t\rightarrow \infty} f(y)\p_{y}W|_{y=0,1}
  \end{align*}
exist (e.g. by Lemma \ref{lem:pWcontrol}) and are non-trivial.
Then for any $s>5/2$,
\begin{align*}
  \sup_{t\geq 0} \|W\|_{H^{s}} = \infty .
\end{align*}
\end{cor}

\begin{proof}
  Suppose to the contrary that for some $s>5/2$, $\|W\|_{H^{s}}$ is bounded uniformly in time. 
  Then, by Lemma \ref{lem:pWcontrol},
  \begin{align*}
    \dt \p_{y}^{2}W|_{y=0,1} = \frac{if}{k}\p_{y}^{2}\Phi + \frac{if'}{k} \p_{y}\Phi |_{y=0,1} = \frac{if}{k^{2}t} \p_{y}W|_{y=0,1} + \mathcal{O}(t^{-1-s}).
  \end{align*}
  Integrating this equation, we thus obtain that 
  \begin{align*}
    \log(t) \lesssim |\p_{y}^{2}W(t)|_{y=0,1} | \leq \|\p_{y}^{2}W(t)\|_{L^{\infty}},
  \end{align*}
  as $t \rightarrow \infty$.
  On the other hand by the Sobolev embedding and the choice of $s>\frac{5}{2}$,
  \begin{align*}
     \|\p_{y}^{2}W(t)\|_{L^{\infty}} \lesssim \|W(t)\|_{H^{s}},
  \end{align*}
  which we supposed to be bounded uniformly in time. This hence yields a contradiction, which proves the desired result. 
\end{proof}

The main result of this section is given by the following Theorem \ref{thm:H52}, which proves that the above restriction is sharp in the sense that stability holds for $s<5/2$.
More precisely, as in Section \ref{sec:stability-h32-}, 
instead of $H^{s}([0,1])$, we consider \emph{periodic} spaces, i.e.
\begin{align*}
  W(t,k,\cdot) \in H^{s-1}(\T), \p_{y}W(t,k,\cdot) \in H^{s-1}(\T),
\end{align*}
which allows us to use both a Fourier characterization and a kernel characterization. 

\begin{thm}
\label{thm:H52}
Let $0<s<1/2$ and let $\omega_{0} \in H^{2}([0,1])$, with vanishing Dirichlet data, $\omega_{0}|_{y=0,1}=0$, and $\omega_{0},\p_{y}\omega_{0},\p_{y}^{2}\omega_{0} \in H^{s}(\T)$. 
Suppose further that $f,g \in W^{3,\infty}(\T)$, that there exists $c>0$ such that 
\begin{align*}
  0<c<g<c^{-1}<\infty,
\end{align*}
 and that 
  \begin{align*}
     \|f\|_{W^{3,\infty}(\T)} L 
  \end{align*}
is sufficiently small.
Then the solution, $W$, of the linearized Euler equations, \eqref{eq:Eulernochmal}, satisfies 
  \begin{align*}
    \|\p_{y}^{2}W(t)\|_{H^{s}(\T)} \lesssim \|\omega_{0}\|_{H^{s}}+ \|\p_{y}\omega_{0}\|_{H^{s}}+\|\p_{y}^{2}\omega_{0}\|_{H^{s}},
  \end{align*}
 uniformly in time.
\end{thm}

\begin{rem}
  \label{rem:H52assumptionsperiod}
Similar to Theorem \ref{thm:H32}, the assumptions on $f$ and $g$ are chosen such that we can apply Proposition \ref{prop:LipHs} to the functions $f$, $g$ and their derivatives $f',f''$ and $g',g''$.
Furthermore, we require 
\begin{align*}
  g^{2}=U'(U^{-1}(\cdot))^{2}
\end{align*}
to be such that we can apply Proposition \ref{prop:ComHs}.

As discussed in Remark \ref{rem:periodicityassumptions}, these assumptions can probably be relaxed to requiring that 
\begin{align*}
  f,g \in W^{4,\infty}([0,1]),
\end{align*}
and that 
\begin{align*}
  |g^{2}(1)-g^{2}(0)| = |(U'(b))^{2}-(U'(a))^{2}|
\end{align*}
is sufficiently small compared to 
\begin{align*}
  \min(g^{2})= \min((U')^{2}) >0.
\end{align*}
\end{rem}

As in the previous section, we split the contributions in the evolution equation into boundary corrections and potentials with zero Dirichlet conditions.
Let thus $W$ be a solution of \eqref{eq:Eulernochmal}, then $\p_{y}^{2}W$ satisfies:
\begin{align}
\label{eq:H2fracnochmal}
\begin{split}
  \dt \p_{y}^{2}W &= \frac{if}{k}(\Phi^{(2)}+H^{(2)}) + \frac{2f'}{ik} (\Phi^{(1)}+H^{(1)}) + \frac{f''}{ik}\Phi ,\\
(-1+(g(\frac{\p_{y}}{k}-it))^{2})\Phi^{(2)} &= \p_{y}^{2}W + [(g(\frac{\p_{y}}{k}-it))^{2}, \p_{y}^{2}] \Phi ,\\
\Phi^{(2)}_{y=0,\pi}&=0 ,
\end{split}
\end{align}
and the \emph{homogeneous correction}, $H^{(2)}$, satisfies 
\begin{align*}
  (-1+(g(\frac{\p_{y}}{k}-it))^{2})H^{(2)} &= 0, \\
  H^{(2)}|_{y=0,\pi}&= \p_{y}^{2}\Phi|_{y=0,\pi}.
\end{align*}

Furthermore, as discussed in the beginning of Section \ref{sec:stability-h32-}, $\Phi^{(1)}$ and $H^{(1)}$ satisfy \eqref{eq:pyW}:
\begin{align}
  \label{eq:pyWnochma}
  \begin{split}
  \dt \p_{y}W &= \frac{if}{k} \p_{y}\Phi + \frac{if'}{k}\Phi, \\
  (-1+(g(\frac{\p_{y}}{k}-it))^{2})\Phi^{(1)} &= \p_{y}W + [(g(\p_{y}-it))^{2}, \p_{y}] \Phi, \\
  \Phi^{(1)}_{y=0,\pi} &= 0 , \\
  H^{(1)}&= \p_{y}\Phi -\Phi^{(1)}, \\
   (t,k,y) &\in  \R \times L(\Z\setminus \{0\}) \times [0,1],
  \end{split}
\end{align}

Considering a decreasing weight $A$ and computing 
\begin{align*}
  \dt ( \langle W, A W \rangle_{H^{s}}+\langle \p_{y} W, A \p_{y}W \rangle_{H^{s}}+\langle \p_{y}^{2}W, A \p_{y}^{2}W \rangle_{H^{s}})=:\dt I(t),
\end{align*}
we show that $I(t)$ is uniformly bounded.

As we have seen in the previous Section \ref{sec:stability-h32-}, the first two terms are non-positive under the conditions of the theorem.

It thus remains to control 
\begin{align*}
  \dt \langle \p_{y}^{2}W, A \p_{y}^{2}W \rangle_{H^{s}}.
\end{align*}

For this purpose, we have to estimate 
\begin{align*}
\tag{elliptic}
  &\langle \frac{if''}{k} \Phi, A \p_{y}^{2}W \rangle_{H^{s}} + \langle \frac{if'}{k} \Phi^{(1)}, A \p_{y}^{2}W \rangle_{H^{s}} +  \langle \frac{if'}{k} \Phi^{(2)}, A \p_{y}^{2}W \rangle_{H^{s}} \\
\tag{boundary}
+& \langle \frac{if'}{k} H^{(1)}, A \p_{y}^{2}W \rangle_{H^{s}} + \langle \frac{if}{k} H^{(2)}, A \p_{y}^{2}W \rangle_{H^{s}}
\end{align*}
in terms of 
\begin{align*}
  \frac{C}{|k|} | \langle W, \dot A W \rangle_{H^{s}}+\langle \p_{y} W, \dot A \p_{y}W \rangle_{H^{s}}+\langle \p_{y}^{2}W, \dot A \p_{y}^{2}W \rangle_{H^{s}}|.
\end{align*}
Requiring $|k|\gg 0$ to be sufficiently large and thus $\frac{C}{|k|}$ to be sufficiently small, then yields the result. 
As in Section \ref{sec:stability-h32-}, the control of the boundary and elliptic contributions is obtained in the following subsections.

\subsection{Boundary corrections}
\label{sec:boundary-corrections-1}
The following two theorems provide a control of the boundary contributions in the proof of Theorem \ref{thm:H52}.
Here, Theorem \ref{thm:h52bd1} controls contributions by $H^{(1)}$ and Theorem \ref{thm:h52bd2} controls contributions by $H^{(2)}$, respectively.
\begin{thm}
\label{thm:h52bd1}
  Let $0<s<1/2$ and let $W,f,g$ as in Theorem \ref{thm:H52}.
   Let further $A$ be a diagonal operator comparable to the identity, i.e.
  \begin{align*}
    A: e^{iny} \mapsto A_{n} e^{iny},
  \end{align*}
  with 
  \begin{align*}
    1 \lesssim A_{n} \lesssim 1,
  \end{align*}
  uniformly in $n$. 
 Then, 
  \begin{align*}
    | \langle A \p_{y}^{2}W, \frac{if'}{k} H^{(1)} \rangle_{H^{s}}| \lesssim \sum_{n} c_{n}(t) <n>^{2s} (|(\p_{y}^{2}W)_{n}|^{2}+|(\p_{y}W)_{n}|^{2}+|W_{n}|^{2}),
  \end{align*}
  where $c_{n} \in L^{1}_{t}$ and $\|c_{n}\|_{L^{1}_{t}}$ is bounded uniformly in $n$.
\end{thm}

\begin{proof}[Proof of Theorem \ref{thm:h52bd1}]
 Combining the approach of Lemma \ref{lem:pWcontrol} and Theorem \ref{thm:H32boundary}, we expand 
  \begin{align*}
     H^{(1)}= \p_{y}\Phi|_{y=0} e^{ikty}u_{1} + \p_{y}\Phi|_{y=1}e^{ikty(y-1)}u_{2}.
  \end{align*}
  We may then estimate
  \begin{align*}
    |\langle A \p_{y}^{2}W, e^{ikty}u_{1} \rangle_{H^{s}}| \lesssim \sum_{n} <n>^{2s}|(\p_{y}^{2}W)_{n}| \frac{1}{|k|<n/k-t>}. 
  \end{align*}

  As by our assumptions $\omega_{0}|_{y=0,1}=0$,  
  \begin{align*}
    \p_{y}\Phi|_{y=0} = \frac{k}{g^{2}(0)}\langle W, e^{ikty}u_{1} \rangle_{L^{2}}
  \end{align*}
  has good decay in time. More precisely, as in Corollary \ref{cor:H2boundaryquadratic}, we integrate by parts twice to obtain control by 
  \begin{align*}
    \left|\p_{y}\Phi|_{y=0,1} \right|= \mathcal{O}(<kt>^{-2}) \|W\|_{H^{2}}.
  \end{align*}
  
Using the $H^{2}$ stability result of \cite{Zill3}, Theorem \ref{thm:citedH1H2result}, we may thus estimate
\begin{align*}
  |\langle A \p_{y}^{2}W, e^{ikty}u_{1} \rangle_{H^{s}}| \lesssim &
<kt>^{-2} \left\| \frac{<n>^{s}}{<n/k-t>^{(1-\gamma)}}(\p_{y}^{2}W)_{n}\right\|_{l^{2}_{n}} \left\| \frac{<n>^{s}}{<n/k-t>^{\gamma}} \right\|_{l^{2}}.
\end{align*}

Choosing $0<\gamma<1$ sufficiently close to $1$ such that $s-\gamma < -\frac{1}{2}$, then yields
  \begin{align*}
    \left\|\frac{<\eta>^{s}}{<n/k-t>^{\gamma}}\right\|_{l^{2}_{n}} = \mathcal{O}(<kt>^{s}).
  \end{align*}

The result thus follows with 
\begin{align*}
  c_{n}(t):= <kt>^{-2+s}<n/k-t>^{-2(1-\gamma)} \in L^{1}_{t}.
\end{align*}
\end{proof}

\begin{thm}
\label{thm:h52bd2}
  Let $0<s<1/2$ and let $A,W,f,g$ as in Theorem \ref{thm:h52bd1}. Then,
  \begin{align*}
    | \langle A \p_{y}^{2}W, \frac{if}{k} H^{(2)} \rangle_{H^{s}}| \lesssim \sum_{n} c_{n}(t) <n>^{2s} (|(\p_{y}^{2}W)_{n}|^{2}+|(\p_{y}W)_{n}|^{2}+|W_{n}|^{2}) ,
  \end{align*}
  where $c_{n} \in L^{1}_{t}$ and $\|c_{n}\|_{L^{1}_{t}}$ is bounded uniformly in $n$.
\end{thm}

\begin{proof}[Proof of Theorem \ref{thm:h52bd2}]
  Following the same approach as in Theorem \ref{thm:H32boundary}, the estimate of 
  \begin{align*}
    |\langle A \p_{y}^{2}W, \frac{if}{k}e^{ity}u_{1} \rangle_{H^{s}} \frac{1}{t} \langle \p_{y}^{2}W,e^{ity}u_{1}  \rangle_{L^{2}}|
  \end{align*}
  is identical up to a change of notation.

  The additional boundary correction in the current case is given by 
  \begin{align*}
    \frac{1}{ikt} \p_{y}W|_{y=0,1}.
  \end{align*}
  While $\p_{y}W|_{y=0,1}$ is not conserved, by Lemma \ref{lem:pWcontrol} it converges as $t \rightarrow \infty$ and is thus in particular bounded.
  This part of the estimate thus also concludes analogously to the proof of Theorem \ref{thm:H32boundary}.
\end{proof}

\subsection{Elliptic regularity}
\label{sec:elliptic-regularity-1}

This subsection's main result is given by the following theorem, which provides control of the elliptic contributions in the proof of Theorem \ref{thm:H52}.
\begin{thm}
  Let $0<s<1/2$ and let $A,f,g,W$ be a as in Theorem \ref{thm:h52bd1}. Then
  \begin{align*}
    & \quad | \langle A \p_{y}W, if \Phi^{(2)} + if' \Phi^{(1)} + if'' \Phi^{(1)} \rangle_{H^{s}}| \\ &\lesssim \sum_{n} c_{n}(t) <n>^{2s} (|(\p_{y}^{2}W)_{n}|^{2}+|(\p_{y}W)_{n}|^{2} + |W_{n}|^{2}),
  \end{align*}
  where $c_{n} \in L^{1}_{t}$ with $\|c_{n}\|_{L^{1}_{t}}$ bounded uniformly in $n$.  
\end{thm}

As in Section \ref{sec:elliptic-control}, Lemma \ref{lem:upperh52} serves to reduce the proof of Theorem \ref{thm:H52} to a fractional elliptic regularity problem. The desired elliptic estimate is then formulated in Proposition \ref{prop:elliph52}, whose prove is further broken down in Lemma \ref{lem:h52hilfslem1} and Lemma \ref{lem:h52hilfslem2}.

\begin{lem}
\label{lem:upperh52}
  Let $0<s<1/2$ and let $A,f,g,W$ as in Theorem \ref{thm:h52bd1}. Then 
  \begin{align*}
    | \langle A \p_{y}W, if \Phi^{(2)} + if' \Phi^{(1)} + if'' \Phi^{(1)} \rangle_{H^{s}}| 
\lesssim & \left( \sum_{n} c_{n}(t) <n>^{2s} |(\p_{y}^{2}W)_{n}|^{2} \right)^{1/2} \\
\big(& \|if\Phi^{(2)}\|_{H^{s}}^{2}+ \|(\frac{\p_{y}}{k}-t) if\Phi^{(2)}\|_{H^{s}}^{2} \\
+& \|if'\Phi^{(1)}\|_{H^{s}}^{2}+ \|(\frac{\p_{y}}{k}-t) if'\Phi^{(1)}\|_{H^{s}}^{2} \\
+&\|if''\Phi\|_{H^{s}}^{2}  + \|(\frac{\p_{y}}{k}-t) if''\Phi\|_{H^{s}}^{2}\big)^{1/2} .
  \end{align*}
\end{lem}

\begin{proof}
  This result is proven in the same way as Lemma \ref{lem:easy} in Section \ref{sec:elliptic-control}.
\end{proof}

The control of 
\begin{align*}
  \|if'\Phi^{(1)}\|_{H^{s}}^{2}+ \|(\frac{\p_{y}}{k}-t) if'\Phi^{(1)}\|_{H^{s}}^{2} 
+\|if''\Phi\|_{H^{s}}^{2}  + \|(\frac{\p_{y}}{k}-t) if''\Phi\|_{H^{s}}^{2} 
\end{align*}
by 
\begin{align*}
  \sum_{n} c_{n}(t) <n>^{2s} (|(\p_{y}W)_{n}|^{2} + |W_{n}|^{2})
\end{align*}
has already been obtained in the previous Section \ref{sec:stability-h32-}. It thus only remains to control 
\begin{align*}
  \|if\Phi^{(2)}\|_{H^{s}}^{2}+ \|(\frac{\p_{y}}{k}-t) if\Phi^{(2)}\|_{H^{s}}^{2},
\end{align*}
which is formulated as the following proposition.
\begin{prop}
\label{prop:elliph52}
  Let $f,g, \omega_{0},W$ be as in Theorem \ref{thm:H52}. Then, 
  \begin{align*}
    \|\Phi^{(2)}\|_{H^{s}}^{2}+ \|(\frac{\p_{y}}{k}-t) \Phi^{(2)}\|_{H^{s}}^{2}\lesssim \sum_{n} c_{n}(t) <n>^{2s} (|(\p_{y}^{2}W)_{n}|^{2}+|(\p_{y}W)_{n}|^{2} + |W_{n}|^{2}).
  \end{align*}

\end{prop}

\begin{proof}[Proof of Proposition \ref{prop:elliph52}] 
We recall that $\Phi^{(2)}$ satisfies \eqref{eq:H2fracnochmal}: 
\begin{align*}
  (1 + (g (\frac{\p_{y}}{k}-it))^{2} ) \Phi^{(2)} &=\p_{y}^{2}W + [(g (\frac{\p_{y}}{k}-it))^{2}, \p_{y}^{2}]\Phi, \\
  \Phi^{(2)}|_{y=0,1}= 0&.
\end{align*}
Using the generic results of Section \ref{sec:elliptic-control} with 
\begin{align*}
  \psi&= \Phi^{(2)}, \\
  R&= \p_{y}^{2}W + [(g (\frac{\p_{y}}{k}-it))^{2}, \p_{y}^{2}]\Phi,
\end{align*}
the result follows if we can obtain a good control of 
\begin{align*}
  \langle \psi, R \rangle_{H^{s}}
\end{align*}
for our specific choice of $R$.

We note that
\begin{align*}
  \langle \psi , \p_{y}^{2}W \rangle_{H^{s}} \lesssim (\|\psi\|_{H^{s}}^{2}+ \|(\frac{\p_{y}}{k}-it)\psi\|_{H^{s}}^{2})^{1/2}
 \left(\sum_{n} <n/k-t>^{-2} <n>^{2s} (|(\p_{y}^{2}W)_{n}|^{2} \right)^{1/2},
\end{align*}
is already of the desired form. 

It thus remains to consider the commutator:
\begin{align*}
  &[(g (\frac{\p_{y}}{k}-it))^{2}, \p_{y}^{2}]\Phi \\
&=: (\frac{\p_{y}}{k}-it)2gg'(\frac{\p_{y}}{k}-it) \Phi^{(1)} +(\frac{\p_{y}}{k}-it)(g^{2})''(\frac{\p_{y}}{k}-it)\Phi  \\
&\quad + (\frac{\p_{y}}{k}-it)2gg'(\frac{\p_{y}}{k}-it) H^{(1)} 
+ h (\frac{\p_{y}}{k}-it) H^{(1)} \\
&\quad + \text{Q},
\end{align*}
where $h$ can be computed in terms of the derivatives of $g$ and $Q$ is composed of terms involving only 
\begin{align*}
  \Phi, \Phi^{(1)}, (\frac{\p_{y}}{k}-it)\Phi, (\frac{\p_{y}}{k}-it)\Phi^{(1)}.
\end{align*}
Thus, 
\begin{align*}
  \langle \psi, Q \rangle_{H^{s}} \lesssim \|\psi\|_{H^{s}} (\|\Phi\|_{H^{s}}^{2} +\|(\frac{\p_{y}}{k}-it)\Phi\|_{H^{s}}^{2}
+\|\Phi^{(1)}\|_{H^{s}}^{2} +\|(\frac{\p_{y}}{k}-it)\Phi^{(1)}\|_{H^{s}}^{2} )^{1/2},
\end{align*}
which, by the $H^{3/2-}$ result, Theorem \ref{thm:H32}, can absorbed
\begin{align*}
  \langle W, \dot A W \rangle_{H^{s}} + \langle \p_{y}W, \dot A \p_{y}W \rangle_{H^{s}} \leq 0.
\end{align*}

The control of the remaining terms is obtained in the following two lemmata.
\end{proof}

\begin{lem}
\label{lem:h52hilfslem1}
  Let $g, \omega_{0},W$ be as in Theorem \ref{thm:H52}. Then, 
  \begin{align*}
    & \quad\langle \Phi^{(2)},(\frac{\p_{y}}{k}-it)2gg'(\frac{\p_{y}}{k}-it) \Phi^{(1)} +(\frac{\p_{y}}{k}-it)(g^{2})''(\frac{\p_{y}}{k}-it)\Phi \rangle_{H^{s}}  \\
&\lesssim  (\|\Phi^{(2)}\|_{H^{s}}^{2} + \|(\frac{\p_{y}}{k}-it)\Phi^{(2)}\|_{H^{s}}^{2})^{1/2} \\
& \quad \cdot
\left(\|\Phi^{(1)}\|_{H^{s}}^{2} + \|(\frac{\p_{y}}{k}-it)\Phi^{(1)}\|_{H^{s}}^{2}+\|\Phi\|_{H^{s}}^{2} + \|(\frac{\p_{y}}{k}-it)\Phi\|_{H^{s}}^{2}\right)^{1/2} \\
& \quad + \sum_{n} c_{n}(t) <n>^{2s} (|(\p_{y}^{2}W)_{n}|^{2}+|(\p_{y}W)_{n}|^{2} + |W_{n}|^{2}),
  \end{align*}
where $c_{n}\in L^{1}_{t}$ and $\|c_{n}\|_{L^{1}_{t}}$ is bounded uniformly in $n$.
\end{lem}

\begin{proof}[Proof of Lemma \ref{lem:h52hilfslem1}]
  Integrating the leading $(\frac{\p_{y}}{k}-it)$ operators by parts, we obtain bulk terms 
  \begin{align*}
    \langle (\frac{\p_{y}}{k}-it)\Phi^{(2)},2gg'(\frac{\p_{y}}{k}-it) \Phi^{(1)} +(g^{2})''(\frac{\p_{y}}{k}-it)\Phi \rangle_{H^{s}},
  \end{align*}
  which can be controlled in the desired manner using Cauchy-Schwarz and Proposition \ref{prop:LipHs} of Section \ref{sec:fract-sobol-space}.

  It thus only remains to control the boundary contributions 
  \begin{align*}
    \sum_{n}\left.\Phi^{(2)}_{n}<n>^{2s} \big( 2gg'(\frac{\p_{y}}{k}-it) \Phi^{(1)} + (g^{2})''(\frac{\p_{y}}{k}-it)\Phi \big)\right|_{y=0}^{1}.
  \end{align*}
  Here we again estimate
  \begin{align*}
    \sum_{n}\Phi^{(2)}_{n}<n^{2s}> \lesssim (\|\Phi^{(2)}\|_{H^{s}}^{2} + \|(\frac{\p_{y}}{k}-it)\Phi^{(2)}\|_{H^{s}}^{2})^{1/2} \|\frac{<n>^{s}}{<n/k-t>}\|_{l^{2}_{n}},
  \end{align*}
  and
  \begin{align*}
    \|\frac{<n>^{s}}{<n/k-t>}\|_{l^{2}_{n}} \lesssim <kt>^{s} .
  \end{align*}

It remains to estimate 
\begin{align*}
  \big( 2gg'(\frac{\p_{y}}{k}-it) \Phi^{(1)} + (g^{2})''(\frac{\p_{y}}{k}-it)\Phi \big)|_{y=0}^{1},
\end{align*}
where we may drop the terms which involve $it$, since $\Phi$ and $\Phi^{(1)}$ satisfy zero Dirichlet boundary conditions.

A good control of $\p_{y}\Phi|_{y=0,1}$ in terms of $\|W\|_{H^{2}}$ has already been obtained in the proof of Corollary \ref{cor:phiboundaryh2}.

It thus remains to control $\p_{y}\Phi^{(1)}|_{y=0,1}$.
Following a similar approach as in Lemma \ref{lem:boundaryestimates}, $\p_{y}\Phi^{(1)}|_{y=0,1}$ can be computed by testing the right-hand-side of the equation against homogeneous solutions:
\begin{align*}
  \langle \p_{y}W + [(g(\frac{\p_{y}}{k}-it))^{2},\p_{y}]\Phi , e^{ikty}u_{1} \rangle_{L^{2}}, \\
  \langle \p_{y}W + [(g(\frac{\p_{y}}{k}-it))^{2},\p_{y}]\Phi , e^{ikt(y-1)}u_{2} \rangle_{L^{2}}.
\end{align*}
In the case of the commutator terms, using integration by parts and the control of 
\begin{align*}
  \|(\frac{\p_{y}}{k}-it)  e^{ikty}u_{1} \|_{L^{2}},
\end{align*}
we estimate by 
\begin{align*}
  \|\Phi\|_{L^{2}} + \|(\frac{\p_{y}}{k}-it)\Phi\|_{L^{2}},
\end{align*}
which is controlled.
In order to estimate the remaining terms involving $\p_{y}W$, we can either use the same approach as in Section \ref{sec:boundary-corrections} and control by 
\begin{align*}
  \sum_{n} c_{n}(t) <n>^{2s} |(\p_{y}W)_{n}|^{2},
\end{align*}
or integrate $e^{ikty}$ by parts to obtain an additional factor $\frac{1}{ikt}$ and estimate by 
\begin{align*}
  \frac{1}{t}\sum_{n} c_{n}(t) <n>^{2s} |(\p_{y}^{2}W)_{n}|^{2}.
\end{align*}
\end{proof}

\begin{lem}
\label{lem:h52hilfslem2}
  Let $g, \omega_{0},W$ be as in Theorem \ref{thm:H52} and let $h \in W^{1,\infty}(\T)$. Then, 
  \begin{align*}
  & \quad |\langle \Phi^{(2)}, (\frac{\p_{y}}{k}-it)2gg'(\frac{\p_{y}}{k}-it) H^{(1)} 
+ h (\frac{\p_{y}}{k}-it) H^{(1)} \rangle_{H^{s}} | \\
&\lesssim \frac{1}{|k|} \left(\| \Phi^{(2)}\|_{H^{s}} + \|(\frac{\p_{y}}{k}-it)\Phi^{(2)}\|_{H^{s}}\right) \frac{1}{t^{2-s}}\|W\|_{H^{2}} . 
  \end{align*}
\end{lem}

\begin{proof}[Proof of Lemma \ref{lem:h52hilfslem2}]
  Using the fact that $H^{(1)}$ solves 
  \begin{align*}
    (-1+ (g(\frac{\p_{y}}{k}-it))^{2}) H^{(1)}=0,
  \end{align*}
  as well as commuting some derivatives, one can express
  \begin{align*}
   \left(\frac{\p_{y}}{k}-it\right)2gg'\left(\frac{\p_{y}}{k}-it\right) H^{(1)}
  \end{align*}
  as 
  \begin{align*}
\left(\frac{\p_{y}}{k}-it\right) h_{1} H^{(1)} + h_{2} H^{(1)},
  \end{align*}
  for some functions $h_{1}, h_{2} \in W^{1,\infty}(\T)$.

  Integrating the $(\frac{\p_{y}}{k}-it)$ by parts and using Proposition \ref{prop:LipHs}, the bulk term is estimated by 
  \begin{align*}
    \left(\| \Phi^{(2)}\|_{H^{s}} + \|(\frac{\p_{y}}{k}-it)\Phi^{(2)}\|_{H^{s}}\right) \|H^{1}\|_{H^{s}},
  \end{align*}
  while the boundary term is estimated in similar way as in the proof of Proposition \ref{prop:ComHs}, by 
  \begin{align*}
    (\| \Phi^{(2)}\|_{H^{s}} + \|(\frac{\p_{y}}{k}-it)\Phi^{(2)}\|_{H^{s}})t^{s} |H^{1}|_{y=0,1}|.
  \end{align*}
  As shown in the proof of Corollary \ref{cor:H2boundaryquadratic}:
  \begin{align*}
    |H^{1}|_{y=0,1}|=\mathcal{O}(t^{-2})\|W\|_{H^{2}}.
  \end{align*}
  Furthermore, 
  \begin{align*}
    H^{(1)}= H^{(1)}|_{y=0}e^{ikty}u_{1} + H^{(1)}|_{y=1}e^{ikt(y-1)}u_{2},
  \end{align*}
  and
  \begin{align*}
    \|e^{ikty}u_{1}\|_{H^{s}}^{2} \lesssim \sum_{n} \frac{<n>^{2s}}{<n-kt>^{2}} \lesssim t^{2s}.
  \end{align*}
Thus,
\begin{align*}
  t^{s} |H^{(1)}|_{y=0,1}| + \|H^{(1)}\|_{H^{s}} \lesssim <t>^{s-2}\|W\|_{H^{2}},
\end{align*}
which concludes the proof.
\end{proof}

We remark that under the conditions of Theorem \ref{thm:H52}, by Theorem \ref{thm:citedH1H2result}, also stability in $H^{2}$ holds. Thus, $\|W\|_{H^{2}}$ can be considered as a given constant.
This then concludes the proof of Theorem \ref{thm:H52}.

Using these improved stability results, in the following Section \ref{cha:anoth-sect-cons}, we revisit the problem of consistency and further consider the implications of these sharp (in)stability results for the nonlinear dynamics. 

Before that, however, in the following section, we further study the formation of singularities at the boundary, the behavior of the homogeneous corrections close the boundary and implications for (in)stability in fractional Sobolev spaces $W^{1,p}([0,1])$.

\section{Boundary layers}
\label{sec:boundary-layers}

Thus far we have seen that $\p_{y}W$ and $\p_{y}^{2}W$, when restricted to the boundary, develop logarithmic singularities as $t \rightarrow \infty$, i.e.
\begin{align*}
  |\p_{y}W|_{y=0,1}| \gtrsim \log(t) .
\end{align*}

While such a point-wise estimate is sufficient to prove instability in $C^{0}$ and thus $H^{s}$ for $s>1/2$, it does not provide a description for $y$ close to the boundary, which would, for example, be useful for the study of $L^{p}$ spaces.

In the following, we therefore analyze the effect of the homogeneous correction on our solution and describe the asymptotic behavior close to the boundary.
Here, for simplicity, we discuss only the evolution of $\p_{y}W$, but all arguments can be adapted to study $\p_{y}^{2}W$ as well.  

Recall that $\p_{y}W$ evolves by \eqref{eq:pyWnochma}:
\begin{align*}
  \dt \p_{y}W = \frac{if}{k}H^{(1)} + \frac{if}{k} \Phi^{(1)} + \frac{if'}{k} \Phi.
\end{align*}
In view of the considerations on linearized Couette flow in Section \ref{sec:const-coeff-model} and as $\Phi^{(1)}$ and $\Phi$ vanish at the boundary and have a good structure, we in the following focus on the asymptotic behavior of 
\begin{align*}
  \frac{if(y)}{k} \int^{T} H^{(1)}(t,y) dt,
\end{align*}
as $T \rightarrow \infty$ and for $y$ close to the boundary.

\begin{lem}
\label{lem:asym1}
  Let $T>1$ and let $u_{1}, u_{2}$ be the solutions of
  \begin{align*}
    (-1+\left(g \frac{\p_{y}}{k}\right)^{2}) u =0,
  \end{align*}
  with boundary values
  \begin{align*}
    u_{1}(0)=u_{2}(1)=1, \\
    u_{1}(1)=u_{2}(0)=0.
  \end{align*}
  Then for any $y \in [0,1]$ 
  \begin{align*}
    \int_{1}^{T}H^{(1)}(t,y) dt = \int_{1}^{T} H^{(1)}(t,0)e^{ikty} dt \ u_{1}(y) + \int_{1}^{T} H^{(1)}(t,1)e^{ikt(y-1)} dt \ u_{2}(y).
  \end{align*}
\end{lem}

\begin{proof}[Proof of Lemma \ref{lem:asym1}]
  It has be shown in the previous sections that
  \begin{align*}
    H^{(1)}(t,y) = H^{(1)}(0,t)e^{ikty} u_{1}(y) + H^{(1)}(1,t)e^{ikt(y-1)} u_{2}(y).
  \end{align*}
  Integrating in time then yields the result.
\end{proof}
In Section \ref{sec:stability-h32-}, we have shown that under $H^{s}$ stability assumptions ,
\begin{align*}
  H^{(1)}(0,t) =\left. \frac{\omega_{0}}{g^{2}}\right|_{y=0} \frac{1}{t} + \mathcal{O}(t^{-1-s}),
\end{align*}
and therefore, for $y=0$,
\begin{align*}
  \int_{1}^{T} H^{(1)}(t,0)e^{ikty} dt |_{y=0} =\frac{\omega_{0}}{g^{2}}|_{y=0}  \int_{1}^{T} \frac{e^{ikty}}{t} dt|_{y=0} + \mathcal{O}(1) \gtrsim \log(T).
\end{align*}
The case $y>0$ is considered in the following lemma, where for convenience of notation we additionally assume that $k>0$.

\begin{lem}
\label{lem:asym2}
  Let $k>0$, then for any $y \geq \frac{1}{2k}$, 
  \begin{align*}
    \int_{1}^{T} \frac{e^{ikty}}{t}dt
  \end{align*}
is bounded uniformly in $T,k$ and $y$.

 For any $0<y<\frac{1}{2k}$,
   \begin{align*}
  \left|\int_{1}^{T} \frac{e^{ikty}}{t}dt\right| \lesssim \min(\log(T), - \log(ky))+ \mathcal{O}(1).
   \end{align*}
 Further restricting to $0<y<\frac{1}{2kT}$, also 
 \begin{align*}
   \Re \left( \int_{1}^{T} \frac{e^{ikty}}{t}dt \right) \gtrsim \log(T) + \mathcal{O}(1).
 \end{align*}
 Letting $T$ tend to infinity, the logarithmic singularity persists:
   \begin{align*}
  \left|\int_{1}^{\infty} \frac{e^{ikty}}{t}dt\right| \gtrsim  -\log(ky)+ \mathcal{O}(1),
   \end{align*}
for $0<y < \frac{1}{k}$.

\end{lem}

\begin{proof}[Proof of Lemma \ref{lem:asym2}]
  By a change of variables, $t \mapsto \tau=kyt$,
  \begin{align*}
    \int_{1}^{T} \frac{e^{ikty}}{t}dt = \int_{ky}^{kyT} \frac{e^{i\tau}}{\tau} d\tau .
  \end{align*}
  Let thus $\frac{1}{2}\leq x_{1}\leq x_{2}$ be arbitrary but fixed, then 
  \begin{align*}
    \int_{x_{1}}^{x_{2}} \frac{e^{i\tau}}{\tau} d\tau = \left.\frac{e^{i\tau}}{i\tau}\right|_{\tau=x_{1}}^{x_{2}} - \int_{x_{1}}^{x_{2}} \frac{e^{i\tau}}{i\tau^{2}} d\tau \lesssim \frac{1}{x_{1}} \leq 2.
  \end{align*}
  Letting $x_{1}=ky$ for $y \geq \frac{1}{2k}$, then proves the first result.
\\

  Let now $0<ky<\frac{1}{2}$. In the case that $kyT>1$, we can choose $x_{1}=1$ and $x_{2}=kyT$ in the above estimate and thus obtain 
  \begin{align*}
    \int_{1}^{kTy} \frac{e^{ikty}}{t}dt = \mathcal{O}(1).
  \end{align*}
It hence suffices to consider 
  \begin{align*}
    \int_{ky}^{\min(kyT,1)} \frac{e^{i\tau}}{\tau} d\tau .
  \end{align*}
  As $\tau \in (0,1)$, 
  \begin{align*}
    0< \cos(1) \leq \Re(e^{i\tau}) \leq 1
  \end{align*}
  does not yield cancellations. Thus, the integral is comparable to 
  \begin{align*}
    \int_{ky}^{\min(kyT,1)} \frac{1}{\tau} d\tau &= \log (\min(kyT,1)) - \log(ky)\\
&= \min(\log(kyT)- \log(ky), -\log(ky)) \\
&= \min(\log(T), -\log(ky)).
  \end{align*}
\\

Letting $T$ tend to infinity, 
\begin{align*}
  \lim_{T\rightarrow \infty}\min(\log(T), -\log(ky)) = -\log(ky),
\end{align*}
which proves the last result.
\end{proof}

We have thus shown that, as $T \rightarrow \infty$, for $y$ close to zero 
\begin{align*}
  \left| \int_{1}^{T}H^{(1)} dt \right| \gtrsim |\log(ky)| + \mathcal{O}(1).
\end{align*}
In particular, while the $L^{\infty}$ norm diverges, for any $1 \leq p <\infty$,
\begin{align*}
  \log(y) \in L^{p}([0,1]),
\end{align*}
and thus no blowup occurs in these spaces.
\\

In view of our stability results for fractional Sobolev spaces, a natural question concerns the behavior of (fractional) $y$ derivatives.
Here we consider 
\begin{align}
\label{eq:hsblowup}
 C_{s}(T,y):= \int_{1}^{T} t^{s} \frac{e^{ikty}}{t} dt,
\end{align}
for $s\in (0,1)$ as a simplified interpolated model between 
\begin{align*}
  \int_{1}^{T} \frac{e^{ikty}}{t} dt,
\end{align*}
and 
\begin{align}
\label{eq:blowups1}
  \frac{d}{dy} \int_{1}^{T} \frac{e^{ikty}}{t} dt = ik \int_{1}^{T} e^{ikty} dt = \frac{e^{ikTy}-e^{iky}}{y}.
\end{align}

We note that, letting $T$ tend to infinity in \eqref{eq:blowups1}, the singularity is of the form 
\begin{align*}
  \frac{1}{y},
\end{align*}
which is not in $L^{p}([0,1])$ for any $1\leq p \leq \infty$.
The intermediate cases $0<s<1$ are considered in the following lemma.

\begin{lem}
\label{lem:asym3}
  Let $0<s<1$ and let $C_{s}(T,y)$ be given by \eqref{eq:hsblowup}. Then 
  \begin{align*}
    C_{s}(T,0) = \frac{T^{s}-1}{s},
  \end{align*}
  and for $0<y<\frac{1}{2k}$,  
  \begin{align*}
    C_{s}(T,y) \lesssim \min(T^{s},(ky)^{-s}) + \mathcal{O}(1).
  \end{align*}
  For $0<y<\frac{1}{2kT}$, also 
  \begin{align*}
    \Re (C_{s}(T,y)) \gtrsim \frac{T^{s}-1}{s} + \mathcal{O}(1).
  \end{align*}
  Letting $T$ tend to infinity, there exists a constant $c \in \C$, which is in general non-trivial, such that 
  \begin{align*}
    C_{s}(\infty,y) = c(ky)^{-s} + \mathcal{O}(1).
  \end{align*}
\end{lem}

\begin{proof}[Proof of Lemma \ref{lem:asym3}]
  For $y=0$, we compute
  \begin{align*}
    \int_{1}^{T} t^{s} \frac{1}{t} dt = \frac{t^{s}}{s}|_{s=1}^{T}= \frac{T^{s}-1}{s}.
  \end{align*}
  Controlling $e^{ikty}$ by its absolute value, this also provides an upper bound for all $y>0$.

  Considering $y>0$, we introduce a change of variables $t \mapsto kyt$
  \begin{align}
    \label{eq:ysblowup}
  \int_{1}^{T} t^{s} \frac{e^{ikty}}{t} dt = (ky)^{-s} \int_{ky}^{kyT} \frac{e^{i\tau}}{\tau^{1-s}} d\tau,
  \end{align}
  which suggests a boundary singularity of the form $\min((ky)^{-s}, T^{s})$.  
  We first estimate 
  \begin{align*}
    \int_{ky}^{kyT} \frac{e^{i\tau}}{\tau^{1-s}} d\tau
  \end{align*}
  from above.
  In the case $1\leq x_{1} \leq x_{2}$, we integrate $e^{i\tau}$ by parts and thus obtain an estimate by 
  \begin{align}
    \label{eq:LettingTtoinfty}
    \left| \int_{x_{1}}^{x_{2}} \frac{e^{i\tau}}{\tau^{1-s}} d\tau \right|\lesssim \frac{1}{x_{1}^{1-s}} \leq 1,
  \end{align}
  which is uniform in $k,y$ and $T$.
  For $x_{1}, x_{2} \leq 1$ it suffices to estimate by the absolute value:
  \begin{align}
    \label{eq:Lettingytozero}
 \left|\int_{x_{1}}^{x_{2}} \frac{1}{\tau^{1-s}} d\tau\right| \lesssim \frac{1}{s} x_{2}^{s} \leq \frac{1}{s}.    
  \end{align}
Hence, by equation \eqref{eq:ysblowup},
\begin{align*}
  \left| \int_{1}^{T} t^{s} \frac{e^{ikty}}{t} dt \right| \lesssim (ky)^{-s} (1+\frac{1}{s}).
\end{align*}
\\

If $ky$ is very small, i.e. $0<y<\frac{1}{kT}$, then again $e^{i\tau}$ does not oscillate and the real part of the integral in \eqref{eq:ysblowup} is comparable to 
  \begin{align*}
    \int_{ky}^{kyT} \frac{1}{\tau^{1-s}} d\tau = \frac{1}{s} (T^{s}-1)(ky)^{s}.
  \end{align*}
  More precisely, we estimate 
  \begin{align*}
    \cos(1)\leq \Re(e^{i\tau}) \leq 1.
  \end{align*}
  We thus obtain a lower bound of 
  \begin{align*}
    \Re (C_{s}(T,y))
  \end{align*}
  by 
  \begin{align*}
   \cos(1) (ky)^{-s}\frac{1}{s} (T^{s}-1)(ky)^{s} =\cos(1) \frac{1}{s} (T^{s}-1).
  \end{align*}
\\

We again consider \eqref{eq:ysblowup}: Then by \eqref{eq:LettingTtoinfty} the limit $T \rightarrow \infty$ exists as an improper integral.
We thus have to show that 
\begin{align*}
  \int_{ky}^{\infty}\frac{e^{i\tau}}{\tau^{1-s}} d\tau 
= c + \mathcal{O}(|ky|^{s})
\end{align*}
for some $c \in \C$, which is in general non-trivial.
By \eqref{eq:Lettingytozero},
\begin{align*}
  \lim_{y \downarrow 0} \int_{0}^{\infty}\frac{e^{i\tau}}{\tau^{1-s}} d\tau =:c,
\end{align*}
 exists.

Splitting and again using \eqref{eq:Lettingytozero},
\begin{align*}
  \int_{ky}^{\infty}\frac{e^{i\tau}}{\tau^{1-s}} d\tau 
= \int_{0}^{\infty}\frac{e^{i\tau}}{\tau^{1-s}} d\tau - \int_{0}^{ky} \frac{e^{i\tau}}{\tau^{1-s}} d\tau
= c+ \mathcal{O}(|ky|^{s}).
\end{align*}

Thus, by equation \eqref{eq:ysblowup}, 
\begin{align*}
 C(\infty,y)=(ky)^{-s} \int_{ky}^{\infty}\frac{e^{i\tau}}{\tau^{1-s}} d\tau = c(ky)^{-s} + \mathcal{O}(1).
\end{align*}
\end{proof}

Letting $T$ tend to infinity, we thus have to control a singularity of the form $y^{-s}$.
\begin{lem}
\label{lem:calc}
  Let $0<s<1$ and let $1 \leq p < \infty$, then 
  \begin{align*}
    y^{-s} \in L^{p}([0,1])
  \end{align*}
  if and only if $p< \frac{1}{s}$.
\end{lem}
\begin{proof}[Proof of Lemma \ref{lem:calc}]
  We explicitly compute
  \begin{align*}
   \|y^{-s}\|_{L^{p}}^{p}= \int_{0}^{1}y^{-sp}= \left. \frac{y^{1-sp}}{1-sp} \right|_{0}^{1}, 
  \end{align*}
  which is finite if and only if $1-sp>0$.
\end{proof}

The above result suggests that, for $1 \leq p < \infty$, 
\begin{align*}
  \sup_{T>1}\left\|\int_{1}^{T}H^{(1)}dt \right\|_{W^{s,p}} 
\end{align*}
is finite for $0<s<\frac{1}{p}$ and infinite for $\frac{1}{p}<s<1$.
For the case $p=2$, we have shown in Section \ref{sec:stability-h32-}, that indeed $s=\frac{1}{2}$ is critical in this sense.

\section{Consistency and implications for nonlinear inviscid damping}
\label{cha:anoth-sect-cons}

A natural question, following the results on linear inviscid damping, concerns the behavior of the full nonlinear dynamics.
In this section, we prove the following three results:
\begin{itemize}
\item Consistency: The linear dynamics are consistent, i.e. the nonlinearity, when evolved by the linear dynamics, is an integrable perturbation (in a less regular space). In the case of non-fractional Sobolev spaces and the infinite periodic channel, this has been addressed in \cite{Zill3}. 
\item Approximation: Supposing nonlinear inviscid damping holds in a space containing $H^{s}$, $s>5$, we show that the solution remains in an $H^{s-5}$ neighborhood of a linear solution (with $U(t,y)$ varying in time) uniformly in time. 
\item Instability: As a consequence, we show that, in a finite periodic channel, the stability result associated to nonlinear inviscid damping can generally not hold in high Sobolev spaces.
  Specifically we show that otherwise $\p_{y}W$ would in general develop a logarithmic singularity at the boundary, which yields a contradiction.
\end{itemize}

The last result in particular implies that a Gevrey regularity result such as in \cite{bedrossian2013inviscid} would have to be heavily modified in the setting of a finite channel.
\\

We first consider consistency, i.e. the evolution of the nonlinear term under the linear dynamics and shows that this would provide a uniformly controlled correction in Duhamel's formula.
For this purpose, we note that the nonlinearity 
\begin{align*}
 v \cdot \omega = \nabla^{\bot} \phi \cdot \nabla \omega, 
\end{align*}
after the change of variables $(x,y)\mapsto (x-tU(y),y)$ is given by 
\begin{align*}
  -(\p_{y}-tU'\p_{x})\Phi \p_{x}W + \p_{x}\Phi(\p_{y}-tU'\p_{x})W = -\p_{y}\Phi \p_{x}W + \p_{x}\Phi \p_{y}W = \nabla^{\bot}\Phi \cdot\nabla W.
\end{align*}
Here, with a slight abuse of notation $\Phi$ and $W$ do not incorporate the change of variables $y \mapsto z:=U^{-1}(y)$.

\begin{thm}[Consistency]
  \label{thm:EConsistency}
  Let $W$ be a solution of the linearized Euler equations, \eqref{eq:Eulernochmal}, in the finite periodic channel, $\T_{L}\times [0,1]$ with
  \begin{align*}
    \int \omega_{0}(x,y) dx \equiv 0,
  \end{align*}
  $f,g \in W^{3,\infty}$ and assume that for some $s\in (2,3)$
  \begin{align*}
    \|W(t)\|_{H^{s}}<C<\infty,
  \end{align*}
  is uniformly bounded (e.g. via Theorem \ref{thm:H52}).
  Then, 
  \begin{align*}
    \|v \cdot \nabla \omega\|_{L^{2}}= \mathcal{O}(t^{-(s-1)}).
  \end{align*}
  In particular, 
  \begin{align*}
    W(t) + \int^{t} \nabla^{\bot} \Phi(\tau) \nabla W(\tau) d\tau 
  \end{align*}
  remains in a bounded neighborhood of $W(t)$ and there exist asymptotic profiles $W_{\pm \infty, con} \in L^{2}$ such that
  \begin{align*}
    W(t) + \int^{t} \nabla^{\bot} \Phi(\tau) \nabla W(\tau) d\tau \xrightarrow{L^{2}} W_{\pm \infty, con},
  \end{align*}
  as $t \rightarrow \pm \infty$.
\end{thm}

\begin{proof}
  Since the change of variables $(x,y)\mapsto (x-tU(y),y)$ is an $L^{2}$ isometry, we obtain
  \begin{align*}
   \|v \cdot \nabla \omega\|_{L^{2}}= \|\nabla^{\bot} \Phi \nabla W\|_{L^{2}}.
  \end{align*}
  As $s>2$ (and we consider two spatial dimensions, $x,y$), we can use a Sobolev embedding to control
  \begin{align*}
    \|\nabla W \|_{L^{\infty}_{xy}(\Omega)} \lesssim \|W\|_{H^{s}}.
  \end{align*}
  It thus suffices to estimate 
  \begin{align*}
    \|\nabla^{\bot} \Phi\|_{L^{2}}.
  \end{align*}
  Taking the $\nabla^{\bot}$ into account and using the damping result, Theorem \ref{thm:lin-zeng}, we obtain
  \begin{align*}
    \|\nabla^{\bot} \Phi\|_{L^{2}} = \mathcal{O}(t^{-(s-1)}) \|W\|_{H^{s}}.
  \end{align*}
  As $s-1>1$, this decay is integrable, which together with the scattering results for $W(t)$, Corollaries \ref{cor:scat1} and \ref{cor:scat2}, concludes the proof. 
\end{proof}

We remark that this consistency result loses regularity and indeed controlling the loss of regularity due to the nonlinearity is one of the main challenges in the nonlinear problem (see \cite{bedrossian2013asymptotic}).

While the linear dynamics are thus consistent in the above sense, higher regularity and how well they approximate the nonlinear dynamics is not answered by the preceding theorem.
\\

In the following, we consider the converse problem, i.e. given a nonlinearly stable solution with inviscid damping, we estimate the effect of the nonlinearity.
For this purpose, we note that the 2D Euler equations 
\begin{align*}
  \dt \omega + v \cdot \nabla \omega &=0, \\
  v &= \nabla^{\bot} \phi , \\
  \Delta \phi &= \omega,
\end{align*}
on either the infinite or finite periodic channel possess a good structure with respect to $x$ averages.
Denote 
\begin{align*}
  \omega&= (\omega- \langle \omega \rangle_{x})+\langle \omega \rangle_{x}= \omega' + \langle \omega \rangle_{x}, \\
  \phi&=  (\phi- \langle \phi \rangle_{x})+\langle \phi \rangle_{x} = \phi' + \langle \phi \rangle_{x} .
\end{align*}
Then, 
\begin{align*}
  \dt \omega' -(\p_{y}\langle \phi \rangle_{x}) \p_{x}\omega' + (\p_{y}\langle \omega \rangle_{x}) \p_{x}\phi' &= (\nabla^{\bot}\phi' \cdot \nabla \omega')' , \\
  \dt \langle \omega \rangle_{x} &= \langle \nabla^{\bot}\phi' \cdot \nabla \omega' \rangle_{x}.
\end{align*}
In analogy to the linear setting, we denote 
\begin{align*}
  -\p_{y}\langle \phi \rangle_{x}=:U(t,y),
\end{align*}
and for the moment restrict our attention to the first equation, considering $U(t,y)$ as given.

In this formulation the Euler equations then read
\begin{align*}
  \dt \omega' + U(t,y) \p_{y} \omega' = (\p_{y}^{2}U(t,y)) \p_{x}\phi' + (\nabla^{\bot}\phi' \cdot \nabla \omega')'.
\end{align*}
Further introducing the volume-preserving change of variables 
\begin{align*}
  (x,y) \mapsto (x-\int_{0}^{t}U(\tau,y) d\tau, y)
\end{align*}
and defining $W,\Phi$ via these coordinates, the \emph{Euler equations in scattering formulation} are given by 
\begin{align}
  \label{eq:ES}
  \tag{ES}
  \dt W = (\p_{y}^{2}U(t,y)) \p_{x}\Phi + \nabla^{\bot}\Phi \cdot \nabla W.
\end{align}
Obtaining a good control of the regularity of $U(t,y)$ as well as appropriate decay is a very hard problem, in particular as the evolution of $U(t,y)$ and $W$ is coupled. In the following theorem, such control is therefore \emph{assumed}.

\begin{thm}[Approximation]
  \label{thm:EApproximation}
  Let $W(t,x,y)$ be a solution of \eqref{eq:ES} and suppose that, for some $s>2$, inviscid damping holds in $H^{s}$ with integrable rates, i.e. suppose that, for some $\epsilon>0$,
  \begin{align*}
    \|\nabla^{\bot} \Phi\|_{H^{s}}= \mathcal{O}(t^{-1-\epsilon})\|W\|_{H^{s+2+\epsilon}}.
  \end{align*}
  Suppose further that $\|W(t)\|_{H^{s+2+\epsilon}}$ is uniformly bounded.
  Then,   
  \begin{align*}
   \|\nabla^{\bot}\Phi \cdot \nabla W \|_{H^{s}} = \mathcal{O}(t^{-1-\epsilon}), 
  \end{align*}
and, in particular,  
\begin{align*}
  \int^{t}_{0}\|\nabla^{\bot}\Phi (\tau) \cdot \nabla W(\tau) \|_{H^{s}} d\tau 
\end{align*}
is bounded uniformly in $t$ and converges as $t \rightarrow \infty$.
\end{thm}

\begin{proof}
  As $s>2$, $H^{s}$ forms an algebra and 
  \begin{align*}
    \|\nabla^{\bot}\Phi \cdot \nabla W \|_{H^{s}} \leq \|\nabla^{\bot}\Phi\|_{H^{s}}\|\nabla W \|_{H^{s}}= \mathcal{O}(t^{-1-\epsilon}),
  \end{align*}
  which proves the result.
\end{proof}

\begin{rem}
  \begin{itemize}
  \item Theorem \ref{thm:lin-zeng} can be extended to provide sufficient conditions for inviscid damping with integrable rates to hold, again assuming sufficient regularity.
    The core problem of inviscid damping is thus again the control of the regularity of $W(t)$.
  \item If $\|W\|_{H^{s+2+\epsilon}}<\delta$ is small, then
    \begin{align*}
      \int^{t}_{0}\|\nabla^{\bot}\Phi (\tau) \cdot \nabla W(\tau) \|_{H^{s}} d\tau = \mathcal{O}(\delta^{2})
    \end{align*}
    is quadratically small. 
    The linearization thus remains valid, but only in a less regular space. For this reason we call this theorem an ``approximation'' result. 
  \item Even if $\|W\|_{H^{s+2+\epsilon}}$ is not small, the nonlinearity yields a bounded contribution. Hence, if
    \begin{align*}
      \|\int^{t}_{0} (\p_{y}^{2}U(\tau,y))\p_{x}\Phi(\tau) d\tau\|_{H^{s}}
    \end{align*}
    grows unboundedly as $t \rightarrow \infty$ , i.e. the linear part is unstable, then, as shown in the following theorem, the nonlinear dynamics can not be stable.
  \end{itemize}
\end{rem}

\begin{thm}[Instability]
  \label{thm:EInstability}
Let $W$ be a solution of \eqref{eq:ES}, $\p_{y}^{2}U(t,y) \in W^{1,\infty}_{y,t}$ and suppose that 
\begin{align*}
  \p_{y}^{2}U (t,y)|_{y=0}> c>0
\end{align*}
for all $t>0$.
Suppose further that for some $k$ 
\begin{align*}
  |\mathcal{F}_{x}(\p_{y}(\nabla^{\bot}\Phi \cdot \nabla W))(t,k,0)| = \mathcal{O}(t^{-1-\epsilon}),
\end{align*}
and 
\begin{align*}
  \mathcal{F}_{x}W(t,k,y)|_{y=0} >c >0,
\end{align*}
for all time.

Then,
\begin{align*}
  | (\mathcal{F}_{x}\p_{y}W)(t,k,0) | \gtrsim \log(t),
\end{align*}
as $t \rightarrow \infty$.

In particular, for any $s>2$, $\|W(t)\|_{H^{s}_{x,y}}$ then can not be bounded uniformly in time.  
\end{thm}

\begin{proof}
Differentiating \eqref{eq:ES} with respect to $y$, we obtain that $\p_{y}W$ satisfies
\begin{align*}
  \dt \p_{y} W = \p_{y} \left(\p_{y}^{2}U(t,y) \p_{x} \Phi\right) + \p_{y} ( \nabla^{\bot}\Phi \cdot \nabla W).
\end{align*}
Restricting to $y=0$ and using that $\p_{x}\Phi$ vanishes on the boundary, as it is assumed to be impermeable, we consider the $k$ Fourier mode.
Then,
\begin{align}
\label{eq:nonlinboundblow}
  \dt \mathcal{F}_{x}\p_{y}W (t,k,0)= \p_{y}^{2}U(t,0) ik (\mathcal{F}_{x} \p_{y}\Phi)(t,k,0) + \mathcal{O}(t^{-1-\epsilon}).
\end{align}
Similar to the previous sections, $\mathcal{F}_{x}\Phi$ solves a shifted elliptic equation:
\begin{align*}
  \left(-k^{2}+ \left(\p_{y}-itk\int^{t}\p_{y}U(\tau,y) d\tau \right)^{2}\right) \mathcal{F}_{x}\Phi = \mathcal{F}_{x}W .
\end{align*}
A homogeneous solution $u$ of this equation is then of the form 
\begin{align*}
 u(t,y)= \exp \left(\int^{t}(U(\tau,y)-U(\tau,0)) d\tau \right) u(0,y).
\end{align*}
By the same argument as in Lemma \ref{lem:boundaryestimates}, $\mathcal{F}_{x}\Phi(t,0)$ can hence be computed in terms of 
\begin{align}
\label{eq:nonlinboundaryblowup}
  \langle \mathcal{F}_{x}W, u(t,y) \rangle_{L^{2}},
\end{align}
where we assumed that 
\begin{align*}
  u(0,0)=1, u(0,1)=0.
\end{align*}
Integrating 
\begin{align*}
  u(t,y)=u(0,y)\frac{1}{\int^{t}\p_{y}U(\tau,y) d\tau} \p_{y}\exp \left(\int^{t}(U(\tau,y)-U(\tau,0)) d\tau \right)
\end{align*}
by parts in \eqref{eq:nonlinboundaryblowup}, then yields a leading order term of the form 
\begin{align*}
\left|\frac{1}{\int^{t}\p_{y}U(\tau,y) d\tau} W|_{y=0} \right| \gtrsim \frac{c}{t}.
\end{align*}
Integrating \eqref{eq:nonlinboundblow} in time thus yields a logarithmic singularity and hence the result.
\end{proof}

We remark that, using a Sobolev embedding, the decay of  
\begin{align*}
   |\mathcal{F}_{x}(\p_{y}(\nabla^{\bot}\Phi \cdot \nabla W))(t,k,0)|
\end{align*}
is a consequence of inviscid damping in a high Sobolev space.
More precisely, supposing that
\begin{align*}
  0<c<\p_{y}U(t,y)<c^{-1}<\infty,
\end{align*}
for $t \geq T$, Theorem \ref{thm:lin-zeng} yields
\begin{align*}
  \|\p_{y}(\nabla^{\bot}\Phi \cdot \nabla W)\|_{L^{\infty}}\leq \|\p_{y}(\nabla^{\bot}\Phi \cdot \nabla W)\|_{H^{2+\epsilon}} &\leq C(\|U'\|_{W^{7,\infty}}) \|W(t)\|_{H^{4+\epsilon}}\|\Phi(t)\|_{H^{4+\epsilon}} \\ &\leq \mathcal{O}(t^{-2+\epsilon})\|W(t)\|_{H^{6}}^{2}.
\end{align*}
Furthermore, restricting \eqref{eq:ES} to the boundary,
\begin{align*}
  \mathcal{F}_{x}W(T,k,0)= \mathcal{F}_{x}\omega_{0}(t,k,0) + \int^{T}_{0} \mathcal{F}_{x} \left(\nabla^{\bot}\Phi \cdot \nabla W \right) (t,k,0) dt.
\end{align*}
If one thus assumes $(\mathcal{F}_{x} \nabla^{\bot}\Phi \cdot \nabla W) (t,k,0)$ to decay with an integrable rate, then $\mathcal{F}_{x}\omega_{0}(t,k,0)$ converges to an in general non-zero limit as $t \rightarrow \infty$.
Considering sufficiently large times $T$,  
\begin{align*}
  \mathcal{F}_{x}W(T,k,0)
\end{align*}
is hence in general bounded away from zero.

The theorem therefore implies that, in the generic case, solutions of \eqref{eq:ES} in a finite periodic channel can not remain bounded in high Sobolev regularity.
In contrast, for the setting of an infinite channel Bedrossian and Masmoudi, \cite{bedrossian2013asymptotic}, establish nonlinear inviscid damping for sufficiently small, highly regular perturbations to Couette flow.
More precisely they require smallness in Gevrey regularity, i.e.
\begin{align*}
  \|\omega_{0}\|_{\mathcal{G}^{\lambda_{0}}}^{2}:= \sum_{k}\int |\tilde{\omega}_{0}(k,\eta)|^{2}e^{2\lambda_{0}|(k,\eta)|^{s}} d\eta \leq \epsilon^{2},
\end{align*}
for some $\frac{1}{2}<s<1$.
The reason for this choice of regularity is given by an analysis of a possible nonlinear frequency cascade, which is estimated to, in the worst case, amount to a loss of Gevrey 2 regularity, i.e. $s=\frac{1}{2}$. 
There are experimental observations, \cite{yu2005fluid}, of some echoes, but it is not clear whether the worst case estimate of the cascade is actually attained. Indeed, the only known lower bound on the required regularity is given by the work of Lin and Zeng, \cite{Lin-Zeng}, who show the existence of non-trivial stationary structures in arbitrarily small $H^{s},s<\frac{3}{2},$ neighbourhoods and that such structures do not exist for $H^{s},s>\frac{3}{2}$.

\bibliographystyle{alpha} \bibliography{citations}

\end{document}